\documentclass[11pt]{amsart}
\usepackage{amsmath,amssymb,latexsym,esint,cite,mathrsfs}
\usepackage{verbatim,wasysym,mathrsfs}
\usepackage[left=2.8cm,right=2.8cm,top=2.86cm,bottom=2.86cm]{geometry}

\usepackage[english]{babel}

\allowdisplaybreaks

\usepackage{color} 
\usepackage{enumerate} 
\theoremstyle{plain}  
\newtheorem{theorem}{Theorem}[section] 
\newtheorem{proposition}[theorem]{Proposition}
\newtheorem{lemma}[theorem]{Lemma}

\theoremstyle{definition}

\theoremstyle{remark}

\newcommand{\R}{\mathbb{R}}  
\newcommand{\ep}{\varepsilon}  

\newcommand{\la}{\lambda}
\newcommand{\equ}[1]{(\ref{#1})}
\newcommand{\ve}{\varepsilon}
\newcommand{\be}{\begin{equation}}
\newcommand{\ee}{\end{equation}}
\newcommand{\bea}{\begin{eqnarray}}
\newcommand{\eea}{\end{eqnarray}}

\let\Section=\section
\def\section{\setcounter{equation}{0}\Section}
\sf

\begin{document}

\title[High energy sign-changing solutions for Coron's problem]{High energy sign-changing solutions for Coron's problem}

\author{Shengbing Deng}
\address{\noindent S. Deng - School of Mathematics and Statistics, Southwest University,
Chongqing 400715, People's Republic of China}\email{shbdeng@swu.edu.cn}

\author{Monica Musso}
\address{\noindent M. Musso - Departamento de Matem\'atica,
Pontificia Universidad Catolica de Chile, Avda. Vicu\~na Mackenna
4860, Macul, Chile}\email{mmusso@mat.puc.cl}

\thanks{The research of the first  author
has been supported by NSFC No. 11501469 and Fundamental Research Funds for the Central Universities XDJK2017B014. The research of the second  author
has been partly supported by Fondecyt Grant 1160135  and Millennium Nucleus Center for Analysis of PDE, NC130017.}

\maketitle

\vskip 0.2cm \arraycolsep1.5pt
\newtheorem{Lemma}{Lemma}[section]
\newtheorem{Theorem}{Theorem}[section]
\newtheorem{Definition}{Definition}[section]
\newtheorem{Proposition}{Proposition}[section]
\newtheorem{Remark}{remark}[section]
\newtheorem{Corollary}{Corollary}[section]

\maketitle

\noindent {\bf Abstract}:
We study the existence of sign changing solutions to the following problem
\begin{eqnarray}\label{0eq1}
\left\{
  \begin{array}{ll}
\Delta u+|u|^{p-1}u=0 \quad  & {\rm in} \quad \Omega_\ep;\\
u=0 \quad  & {\rm on} \quad\partial \Omega_\ep,
  \end{array}
  \right.
\end{eqnarray}
where $p=\frac{n+2}{n-2}$ is the critical Sobolev exponent  and $\Omega_\ep$ is a bounded smooth domain in $\R^n$, $n\geq 3$, with the form $\Omega_\ep=\Omega\backslash B(0,\ep)$ with $\Omega$  a smooth bounded domain containing the origin $0$ and $B(0,\ep)$  the ball centered at the origin with radius $\ep >0$. We construct a new type of sign-changing solutions with high energy to problem \eqref{0eq1}, when the parameter $\ep$ is small enough.

\vspace{3mm}

\maketitle

\section{Introduction}\label{intro}

Let $D$  be a smooth bounded domain in $\mathbb{R}^n$, with $n \geq 3$, and let $p=\frac{n+2}{n-2}$ be the critical Sobolev exponent, namely the exponent for which  the Sobolev embedding $H_0^1(D)\hookrightarrow L^{p+1}(D)$ ceases to be compact.
A  problem which has been widely studied in the last 30 years concerns the existence, multiplicity, and qualitative properties for positive or sign-changing solutions
to
\begin{eqnarray}\label{feq1}
\left\{
  \begin{array}{ll}
\Delta u+|u|^{p-1}u=0 \quad  & {\rm in} \quad D;\\
u=0 \quad  & {\rm on} \quad\partial D.
  \end{array}
  \right.
\end{eqnarray}
Solvability for Problem \eqref{feq1} is not a trivial issue, since it strongly depends on the geometry of $D$.
Let us briefly summarize some classical results.
A direct consequence of Pohozaev's indentity \cite{ph} is that Problem (\ref{feq1}) has no positive solutions when the domain $D$ is strictly star-shaped.
On the other hand, if $D$ is an annulus, then Kazdan and Warner \cite{kw} showed that solvability for problem (\ref{feq1}) is restored: a solution is found as critical point of the energy functional associated to the problem, and the required compactness for the functional is obtained thanks to radially symmetry. A surprising result by Coron  showed later on  that symmetry is not really needed to have solvability: in his classical work \cite{coron} he proves  the existence of a positive solution to (\ref{feq1}) in the case in which $D$ has a small  (not necessarily symmetric) hole. In literature, the name {\it Coron's problem} is referred to problem \eqref{feq1}, when the domain $D$ has a hole. The result by Coron  was then generalized by Bahri and Coron in \cite{baco}, where the authors showed that, under the assumption that some homology group of $D$ with coefficients in $Z_2$ is not trivial then problem (\ref{feq1}) has at least one positive solution. Multiplicity result for positive solutions to (\ref{feq1}) is obtained in \cite{rey}, where the situation of a domain
 $D$ with several holes is treated.
On the other hand, existence and multiplicity of positive solutions has been established also in contractible domains, we refer to the works of Dancer \cite{dancer}, Ding \cite{ding}, and Passaseo \cite{pass1}-\cite{pass2}, among others.

The study of solutions for elliptic problems with critical nonlinearity which change sign
has received the interest  of several authors in the last years, see for instance \cite{BMP,CW},\cite{vaira,vaira2}, and references therein.
Here we focus our interest in existence and qualitative properties of sign-changing solutions to \eqref{feq1} for domains $D$ which have a hole, that is in the {\it Coron's setting}. The first result available in literature is the one contained in \cite{mupi2}, where a large number
of sign changing solutions to (\ref{feq1}) in the presence of a single small hole has been proved. To be more precise,  the authors assume that the domain $D$ has the form $D=\Omega \backslash B(0,\ep)$, where $\Omega$ is a smooth bounded domain containing the origin, and it is symmetric with respect to the origin, while the hole is given by $B(0,\ep)$, a round ball with size $\ep$. In this situation, they prove that, for any given integer $k$ there exist $\ep_k>0$ so that, for any $\ep \in (0, \ep_k)$ a sign-changing solution to \eqref{feq1} exists, and it has the shape of a superposition of exactly $k$ bubbles, centered at $0$. A bubble is the function
\begin{equation}\label{bub}
U(y) = \alpha_n \left( {1 \over 1+ |y|^2} \right)^{n-2 \over 2}, \quad  \alpha_n = \left[  n (n-2) \right]^{n-2 \over 4},
\end{equation}
and it solves
\begin{equation}\label{bub1}
\Delta u + u^p = 0, \quad {\mbox {in}} \quad \R^n.
\end{equation} It is well known that all positive solutions to \eqref{bub1} are given by the function \eqref{bub} and
any translation, and proper dilation of it, that is by
\begin{equation}\label{bub2}
U(x-\xi) , \quad {\mbox {for}} \quad \xi \in \R^n, \quad {\mbox {and}} \quad \lambda^{-{n-2 \over 2}} U({x\over \lambda} ), \quad {\mbox {for}} \quad \lambda >0,
\end{equation}
see
\cite{aubin,obata,talenti}. The solutions found in \cite{mupi2} are bubble-tower, with $k$ bubbles of alternating sign, which at main order look like
$$
  u_\ep (x) \sim \sum_{j=1}^k (-1)^{j+1} \lambda_j^{-{n-2 \over 2} } U\left({x\over \lambda_j}\right) , \quad \lambda_j \sim \ep^{2j-1 \over 2k},
$$
as $\ep \to 0^+$. The energy of these solutions has the asymptotic expansion
$$
e(u_\ep ) = k S_n  + O \left( \ep^{n-2 \over 2} \right), \quad {\mbox {as}} \quad \ep \to 0^+, \quad {\mbox {where}} \quad S_n =\left({1\over 2} - {1\over p+1} \right) \, \left( \int_{\R^n} |\nabla U|^2 \, dx \right).
$$
We recall the explicit definition of the energy $e$, given by
\begin{equation}\label{bub2}
e(u) = {1\over 2} \int_D |\nabla u|^2 - {1\over p+1} \int_D |u|^{p+1}.
\end{equation}
Substantial improvement of this result was obtained in \cite{GMP} where the assumption of symmetry was removed. See also \cite{mupi}.
A different construction of sign-changing solutions for \eqref{feq1} on a domain with a small hole  has been obtained in \cite{CMP}: in this case, solutions look like the sum of a positive bubble centered inside the shrinking hole, and a number of sign-changing bubbles centered inside the domain, far from the hole. A common feature of the constructions in \cite{mupi,mupi2,GMP,CMP} above described is that the building blocks that constitute the core of the shape of the solutions are given by the positive solutions to \eqref{bub1}, which are completely classified by the functions given by  \eqref{bub} and \eqref{bub2}.

\smallskip

The aim of this work is to produce a new, different construction of sign-changing solutions to problem (\ref{feq1}), which are build upon a limit profile given by an explicit {\it sign-changing} solution to
\begin{equation}\label{eqqq}
\Delta u + |u|^{p-1} u = 0 , \quad {\mbox {in}} \quad \R^n.
\end{equation}
In  \cite{dmpp1} it is proven the existence of a sequence of finite-energy, sign-changing solutions $Q_k$ to \eqref{eqqq}, with a {\it crown-like} shape and energy of size $(k+1) S_n (1+ o(1) )$, as $k \to \infty$. For any $k$ large, $Q_k$ is given by
\begin{equation}
\label{sol}
Q_k(x) =U_* (x) +\tilde  \phi (x), \quad U_* (x) =  U(x) - \sum_{j=1}^k U_j (x),
\end{equation}
where
$\tilde \phi$ is  smaller than $U_*$, in some sense that we precise later.
The function $U$ is given by \eqref{bub} and  $U_j$ are defined as
\begin{equation*}
\label{basic}
 U_j (x) = \mu_k^{-{n-2 \over 2}} U(\mu_k^{-1} (x-\xi_j )).
\end{equation*}
For any integer $k$ large, the $k$ points $\xi_j$, $j=1, \ldots , k$ are vertices of a regular polygon of $k$ edges, contained in the $(x_1 , x_2)$-plane, given by
\begin{equation*} \label{parameters1}
 \xi_j= \sqrt{1-\mu_k^2} \, ( \cos \theta_j , \sin \theta_j ,  0) \in \R^2 \times \R^{n-2}, \quad  \theta_j = {2\pi(j-1) \,  \over k} \,
\end{equation*}
and the parameter $\mu_k$ is defined as
\begin{equation} \label{parameters}
\Big[  \sum_{j= 1}^k {1 \over (1-\cos \theta_j )^{n-2 \over 2}} \Big] \, \mu_k^{n-2 \over 2} =   1+ O ({1\over k})   , \quad {\mbox {for}} \quad k \to \infty.
\end{equation}
Observe that the previous definition gives
\begin{equation} \label{parametersas}
\mu_k\sim\left\{
\begin{array}{ll}
k^{-2}, \quad  &\mbox{ if }\quad n\geq 4,\\[1mm]
k^{-2} |\log k|^{-2}, \quad  &\mbox{ if }\quad n=3.
\end{array}
\right.
\end{equation}
In other words, $Q_k$ is a crown-like function, with a central positive bubble, centered at the origin in $\R^n$, and a large number of negative
copies of bubbles, centered at the vertices of a regular $k$-polygon sit in the $(x_1 , x_2 , 0 , \ldots , 0)$-plane, each one  with a very sharp profile, as $k \to \infty$.
A property of these solutions that is central for our construction to work is that they  are non-degenerate, as proved in \cite{mw}. In the class of bounded functions, the linearized operator around $Q_k$
$$ L(\phi ) = \Delta \phi + p |Q_k|^{p-1} \phi $$
has a kernel of finite dimension, equals to $3n$. Section \ref{appa} is devoted to give a precise description of the main properties of these solutions, including their non-degeneracy.

\smallskip

In this paper, we construct sign changing solutions to problem (\ref{feq1})
using the function $Q_k$ as main building block, thus generating a new type of sign-changing solutions, different from the ones already known in literature \cite{mupi,mupi2,GMP,CMP}.

 Let $\Omega$ be a smooth bounded domain in $\mathbb{R}^n$, $\ep >0$ and
 $\Omega_\ep =\Omega\backslash B(0,\ep)$.
We consider the following problem
\begin{eqnarray}\label{eq1}
\left\{
  \begin{array}{ll}
\Delta u+|u|^{p-1}u=0 \quad  & {\rm in} \quad \Omega_\ep;\\
u=0 \quad  & {\rm on} \quad\partial \Omega_\ep.
  \end{array}
  \right.
\end{eqnarray}

Our result states the following

\begin{theorem}\label{main}
There exists an integer $k_0$ such that for any integer $k\geq k_0$, there exists $\ep_k$, such that for any $\ep\in(0,\ep_k)$, Problem (\ref{eq1}) has a sign changing solution $u_\ep$, satisfying
$$
u_\ep(x)=d_\ep^{-\frac{n-2}{2}} \ep^{-\frac{n-2}{4}}Q_k \left(\frac{x}{\ep d_\ep^2}\right)(1+o(1))
$$
where $o(1)\to0$ uniformly as $\ep\to0$, and $d_\ep\to d_0$ with $d_0$ is a positive constant. Moreover
$$
J_\ep (u_\ep) = (k+1) S_n + O(\ep^{n-2 \over 2} ), \quad {\mbox {as}} \quad \ep \to 0^+,
$$
where
$$
J_\ep (u_\ep) = \frac{1}{2}\int_{\Omega_\ep}|\nabla u |^2dx-\frac{1}{p+1}
\int_{\Omega_\ep}|u |^{p+1}dx.
$$
\end{theorem}

\medskip

This paper is organized as follows. In Section \ref{appa} we recall the main properties of the functions $Q_k$, that are used in the rest of the paper. In Section \ref{first}, we define a first approximation  to problem
(\ref{eq1}). We give the expansion of the energy functional at the first approximation in Section \ref{energys}.
The proof of Theorem \ref{main} is contained in Section \ref{exst}. Section \ref{finitep} is
devoted to solve a  linear problem, and Section \ref{nonlinear} is devoted to solve a  nonlinear problem. Further properties on the functions $Q_k$, which are new, are reported  in the Appendix A.

\section{Building blocks: Sign-Changing solutions $Q_k$ in $\mathbb{R}^n$}\label{appa}

The solutions predicted by Theorem \ref{main} are constructed as small perturbation of an initial approximation. This initial approximation is build using the entire, finite energy, sign changing solution $Q_k$ for the problem \eqref{eqqq},  which are mentioned in the Introduction. The existence of such solutions is proven in \cite{dmpp1,dmpp2}: this is a sequence of solutions, defined for any integer $k$ sufficiently large.
If we define  the energy by
\begin{equation} \label{energy1}
E(u) = {1\over 2} \int_{\R^n} |\nabla u|^2 \, dy - {1\over p+1} \int_{\R^n} |u|^{p+1} \, dy ,
\end{equation}
we have
$$
E (Q_k ) =  \left\{ \begin{matrix}   (k+1) \, S_n \, \left( 1+ O(k^{2-n})  \right)  & \hbox{ if } n\geq 4\, , \\
& \\     (k+1) \, S_3 \, \left( 1+ O(k^{-1} |\log k|^{-1} \right) & \hbox{ if } n=3,
 \end{matrix}
  \right.
$$
as $k\to \infty$,
where $S_n$ is a positive constant, depending on $n$. The solution $Q=Q_k$ decays at infinity like the fundamental solution, namely
\begin{equation}
\label{vanc1}
\lim_{|y| \to \infty} |y|^{n-2} \, Q_k(y) = \beta_n \left( 1+ c_k \right)
\end{equation}
where $\beta_n$ is a positive number and
$$
 c_k = \left\{ \begin{matrix}   O({1\over k^{n-3}})  & \hbox{ if } n\geq 4\, , \\ & \\     O( |\log k|^{-1} )  & \hbox{ if } n=3, \end{matrix}
  \right. \ \quad {\mbox {as}} \quad k \to \infty.
$$
Furthermore, the solution $Q=Q_k$ has a positive global non degenerate maximum at $y=0$. To be more precisely we have
\begin{equation}
\label{at0}
Q(y) = \left[ n (n-2) \right]^{n-2 \over 4} \, \left ( 1- {n-2 \over 2} |y|^2 + O(|y|^3 ) \right) \quad {\mbox {as}} \quad |y|\to 0.
\end{equation}
Another property for  the solution $Q=Q_k$ is that it is invariant under rotation of angle ${2\pi \over k}$ in the $y_1 , y_2$ plane, namely
\begin{equation}
\label{sim00}
Q( e^{2\pi \over k} \bar y ,y' ) = Q(\bar y , y' ), \quad \bar y= (y_1, y_2) , \quad y'= (y_3, \ldots , y_n ).
\end{equation}
It is even in the $y_j$-coordinates, for any $j=2, \ldots , n$
\begin{equation}
 Q (y_1,\ldots,y_j, \ldots, y_{n}  ) =  Q (y_1,\ldots,-y_j, \ldots, y_{n}  ),\quad  j=2,\ldots,n.
\label{sim22}\end{equation}
It respects invariance under Kelvin's transform:
\begin{equation}
 Q (y)  =   |y|^{2-n} Q (\frac{y}{|y|^{2}} ).
\label{sim33}\end{equation}
The function $ \tilde \phi $ in \equ{sol}  can be further decomposed. Let us introduce  some cut-off functions $\zeta_j$
to be defined as follows.
Let $\zeta(s)$ be a smooth function such that $\zeta(s) = 1$ for $s<1$ and $\zeta(s)=0$ for $s>2$. We also let $\zeta^-(s) = \zeta(2s)$.
Then we set
$$
 \zeta_j(y) = \left\{ \begin{matrix}   \zeta(\, k\eta^{-1} |y|^{-2} |( y  -\xi|y|)\,  |\, )  & \hbox{ if } |y|> 1\, , \\ & \\     \zeta{ (\,k \eta^{-1}\,|y-\xi|\, )}   & \hbox{ if } |y|\le  1\, , \end{matrix}
  \right. \
$$
in such a way that that
$$
\zeta_j( y) = \zeta_j( y/|y|^2).
$$
The function $\tilde \phi$  has the form
\begin{equation} \label{decphi}
\tilde  \phi\  =\  \sum_{j=1}^k  \tilde \phi_j  + \psi.
\end{equation}
In the decomposition \equ{decphi} the functions $\tilde \phi_j$, for $j>1$, are defined in terms of $\tilde \phi_1$
\begin{equation}
\tilde  \phi_j (\bar y, y')= \tilde \phi_1( e^{\frac{2\pi j} k i} \bar y, y'), \quad j=1,\ldots, k-1.
\label{sim11}\end{equation}
We have that
\begin{equation}
\label{estpsi}
 \| \psi \|_{n-2} \leq C k^{1-{n\over q}} \quad {\mbox{if}} \quad n\geq 4, \quad  \| \psi \|_{n-2} \leq {C \over \log k} \quad {\mbox{if}} \quad n=3,
\end{equation}
where $q> {n\over 2}$, and
\begin{equation}
 \|\phi\|_{n-2} :=
\|\,(1+ |y|^{n-2}) \phi \, \|_{L^\infty(\R^n)} \ .
\label{nomstar1}\end{equation}
On the other hand, if we rescale and translate the function $\tilde \phi_1$
\begin{equation}
\label{phiriscalata}
\phi_1 (y) =\mu^{n-2 \over 2}  \tilde \phi_1 (\xi_1 + \mu y )
\end{equation}
we have the validity of the following estimate for $\phi_1$
\begin{equation}
\label{estphi1}
\|  \phi_1 \|_{n-2} \leq C k^{-{n\over q}}  \quad {\mbox {if}} \quad n\geq 4, \quad \|  \phi_1 \|_{n-2} \leq {C \over k \log k}  \quad {\mbox {if}} \quad n= 3.
\end{equation}
The description of the solution $Q$ in \eqref{sol} is thus quite accurate. For later purpose, we observe that the region where $Q$ changes sign is well understood: there exists $0<R_1 <1 <R_2$, positive, so that
\begin{equation}\label{positive}
Q(y) >0 , \quad {\mbox {for}} \quad |y| \leq R_1, \quad |y| \geq R_2.
\end{equation}
In \cite{mw}, is was proved that these solutions are {\it non degenerate}.
That is, fix one solution $Q=Q_k$ of problem \equ{eqqq} and define  the linearized equation around $Q$ as follows
\begin{equation}
\label{defL}
L(\phi ) = \Delta \phi + p |Q|^{p-1}  \phi .
\end{equation}
The invariances \equ{sim00}, \equ{sim22}, \equ{sim33}, together with the natural invariance of any solution
to \equ{eqqq} under translation (if $u$ solves \equ{eqqq} then also $u (y+\zeta)$ solves \equ{eqqq} for any $\zeta \in \R^n$) and under dilation (if $u$ solves \equ{eqqq} then $\la^{-{n-2 \over 2} } u(\la^{-1} y)$ solves \equ{eqqq} for any $\la >0$) produce some {\it natural} functions $\varphi$  in the kernel of $L$, namely
$$
L(\varphi ) = 0.
$$
These are the $3n$  linearly independent functions we introduce next:
\begin{equation}
\label{capitalzeta0}
z_0 (y) = {n-2 \over 2} Q(y) + \nabla Q (y) \cdot y ,
\end{equation}
\begin{equation}
\label{capitalzetaj}
z_\alpha (y)  = {\partial \over \partial y_\alpha  } Q(y) , \quad {\mbox {for}} \quad \alpha=1, \ldots , n,
\end{equation}
and
\begin{equation}
\label{capitalzeta2}
z_{n+1} (y) = -y_2  {\partial \over \partial y_1 } Q(y) + y_1  {\partial \over \partial y_2  } Q(y),
\end{equation}
\begin{equation}
\label{chico1}
z_{n+2} (y) = -2 y_1 z_0 (y) + |y|^2 z_1 (y) , \quad z_{n+3} (y) =  -2 y_2 z_0 (y) + |y|^2 z_2 (y)
\end{equation}
and, for $l=3, \ldots , n$
\begin{equation}
\label{chico2}
z_{n+l+1} (y) = -y_l z_1 (y) + y_1 z_l (y), \quad z_{2n+l-1} (y) =  -y_l z_2 (y) + y_2 z_l (y).
\end{equation}
One has
\begin{equation}\label{neve}
L(z_\alpha ) = 0 , \quad {\mbox {for all}} \quad \alpha = 0 , 1 , \ldots , 3n-1.
\end{equation}
To show \eqref{neve}, let us introduce the following operator:
for any set of parameters $A=(\lambda,\xi , a,\theta)\in \mathbb{R}^+\times \R^n \times \R^2\times \R^{2n-3}$ and for any function $f : \R^n \to \R$,  we define the following function
\begin{equation}\label{thetaa0}
\Theta_A [f] (x) = \la^{-{n-2 \over 2}} \left| {x - \xi  \over |x - \xi |} - a {|x - \xi | \over \la }  \right|^{2-n} \,
 f\left( { R_\theta \left( {x - \xi \over \la }  - a |{x - \xi \over \la }|^2 \right) \over |  {x - \xi \over |x-\xi |} - a {|x - \xi| \over \la }|^2 }\right).
\end{equation}
When $f= Q=Q_k$ is the non-degenerate solution to problem \equ{eqqq}, for simplicity we define
\begin{equation}
\label{thetaa}
\Theta_A (x) =\Theta_A [Q] (x) .
\end{equation}
In \cite{DKM} it is proven that for  any choice of $A$, the function $\Theta_A$ is still a solution of \equ{eqqq}, namely
$$
\Delta \Theta_A + |\Theta_A |^{p-1} \Theta_A = 0 , \quad {\mbox {in}} \quad \R^n.
$$
Observe now that
\begin{equation}\label{fin1}
 \left( {\partial \over \partial \la} \Theta_{(\la , 0 , 0 , 0)} (x) \right)_{\la =1} = - z_0 (x) ,
 \quad  \left( {\partial \over \partial \xi_j} \Theta_{( 1, \xi , 0 , 0)} (x) \right)_{\xi =0} = - z_j (x), \quad j=1, \ldots , n
\end{equation}
\begin{equation}\label{fin2}
 \left( {\partial \over \partial a_1} \Theta_{(1 , 0 , a , 0)} (x) \right)_{a=0} =  z_{n+2} (x) ,\quad
  \left( {\partial \over \partial a_2} \Theta_{(1 , 0 , a , 0)} (x) \right)_{a=0} =  z_{n+3} (x).
 \end{equation}
 Identities \eqref{fin1} and \eqref{fin2} say that $z_0$ is related to the invariance of Problem \equ{eqqq} with respect to dilation $\la^{-{n-2 \over 2} } Q (\la^{-1} y)$, $z_i$, $i=1, \ldots , n$,   are related to the invariance of Problem \equ{eqqq} with respect to translation $Q (y+\zeta )$,
 $z_{n+2}$ and $z_{n+3}$ defined in \equ{chico1} are related to the invariance of Problem \equ{eqqq} under Kelvin transformation \equ{sim33}.
We shall see next that the function $z_{n+1}$ defined in \equ{capitalzeta2} is related to the invariance of $Q$ under rotation in the
$(y_1 , y_2)$ plane, while the functions defined in \equ{chico2} are related to the invariance under rotation in
the $(y_1 , y_l)$ plane and in the $(y_2 , y_l)$ plane respectively.

Let us be more precise. Denote by $O(n)$ the orthogonal group of $n\times n$  matrices $P$ with real coefficients, so that $P^T P = I$, and
by $SO(n)  \subset O(n) $ the special orthogonal group of all matrices in $O(n)$ with $det P=1$. $SO(n)$ is the group of all rotations in $\R^n$, it is a compact
group, which can be identified with a compact set in $\R^{n (n-1) \over 2}$.
Consider the sub group  $\hat S$ of $SO(n)$ generated by  rotations in the $(x_1 , x_2)$-plane, in the $(x_j , x_\alpha)$-plane, for any $j=1,2$ and $\alpha = 3, \ldots , n$. We have that $\hat S$ is compact and can be identified with a compact manifold of dimension  $2n-3$, with no boundary.
In other words, there exists a smooth injective map
$\chi : \hat S \to \R^{{n(n-1) \over 2}}$ so that $\chi ( \hat S)$ is a compact manifold of dimension $2n-3$ with no boundary and $\chi^{-1} : \chi (\hat S ) \to \hat S$ is a smooth parametrization of $\hat S$  in a neighborhood of the Identity. Thus  we write
$$
\theta \in \mathcal{O} =  \chi (\hat S ), \quad
R_\theta = \chi^{-1} (\theta )
$$
where $\mathcal{O}$ is a compact manifold of dimension $2n-3$ with no boundary and $R_\theta$
denotes a rotation in $\hat S$. Let $\theta=(\theta_{12},\theta_{13},\ldots,\theta_{1n},\theta_{23},\ldots,\theta_{2n})$,
and we write
$$
R_\theta= P_{12}(\theta_{12})P_{13}(\theta_{13})P_{14}(\theta_{14})\cdots P_{1n}(\theta_{1n})
P_{23}(\theta_{23})P_{24}(\theta_{24})\cdots P_{2n}(\theta_{2n}),
$$
where   $P_{ij}(\theta)$  is the  rotation in the $(i,j)-$plane of an angle $\theta$,
\begin{eqnarray*}
P_{ij}(\theta) =   \left(
                           \begin{array}{ccccccccc}
                             1      & \cdots &   0    &   0    & \cdots &    0   & 0      &\cdots & 0 \\
                             \vdots & \ddots & \vdots & \vdots & \ddots & \vdots & \vdots &\ddots & \vdots \\
                             0      & \cdots & \cos \theta & 0 & \cdots & 0 & -\sin \theta &\cdots & 0 \\
                             0 & \cdots & 0 & 1 & \cdots & 0 & 0 &\cdots & 0 \\
                             \vdots & \ddots & \vdots & \vdots & \ddots & \vdots & \vdots &\ddots & \vdots \\
                             0 & \cdots & 0 & 0 & \cdots & 1 & 0 &\cdots & 0 \\
                             0      & \cdots & \sin \theta & 0 & \cdots & 0 & \cos \theta &\cdots & 0 \\
                              \vdots & \ddots & \vdots & \vdots & \ddots & \vdots & \vdots &\ddots & \vdots \\
                             0 & \cdots & 0 & 0 & \cdots & 0 & 0 &\cdots & 1 \\
                           \end{array}
                         \right),\qquad i<j.
\end{eqnarray*}
We have
\begin{equation}\label{fin3}
 \left( {\partial \over \partial \theta_{12}} \Theta_{(1 , 0 , \theta , 0)} (x) \right)_{\theta =0} =  z_{n+1} (x)
 \end{equation}
and, for any $l=3, \ldots , n$,
\begin{equation}\label{fin4}
 \left( {\partial \over \partial \theta_{1 l }} \Theta_{(1 , 0 , \theta , 0)} (x) \right)_{\theta =0} =  z_{n+l +1} (x) , \quad
 \left( {\partial \over \partial \theta_{ 2 l }} \Theta_{(1 , 0 , \theta , 0)} (x) \right)_{\theta =0} =  z_{n+l-1} (x)
 \end{equation}
Thus we have the validity of \eqref{neve} as direct consequence of \eqref{fin1}, \eqref{fin2}, \eqref{fin3}, \eqref{fin4}.

In \cite{mw} it is proven that there exists a sequence of solutions  $Q = Q_k $ of the form \eqref{sol} which are  non degenerate in the sense that
\begin{equation}
\label{nondeg} {\mbox {Kernel}} (L) = {\mbox {Span}} \{ z_\alpha \, : \, \alpha = 0 , 1 , 2, \ldots , 3n-1 \},
\end{equation}
or equivalently, any bounded (or any solution in ${\mathcal D}^{1,2}$) of $L(\varphi ) = 0 $ is a linear combination of the functions $z_\alpha $, $\alpha = 0 , \ldots , 3n-1$. The non-degeneracy of $Q$ is a crucial property for our construction.

\section{The first approximation solution}\label{first}

Let $\eta >0$ be a fixed and small number and let us introduce a set of parameters $A=(\lambda,\xi , a,\theta)\in \mathbb{R}^+\times \R^n \times \R^2\times \R^{2n-3}$
with the properties that
\begin{eqnarray}\label{txiaztea}
\lambda=d\sqrt{\ep},\quad \mbox{with}\ \ \eta<d<\frac{1}{\eta},\quad \mbox{for\ some\ fixed}\ \eta>0,
\end{eqnarray}
\begin{eqnarray}\label{txiaztea0}
\xi = \la \tau \in \R^n , \quad \mbox{with}\ \ | \tau | < \eta
\end{eqnarray}
\begin{eqnarray}\label{txiaztea1}
 a\in \mathbb{B}:=\left\{a=(a_1,a_2,0,\ldots,0)\in \R^n\ :\ |a|<\frac{1}{2}\right\},
\end{eqnarray}
and
\begin{eqnarray}\label{ztea2n3}
\theta=(\theta_{12},\theta_{13},\ldots,\theta_{1n},\theta_{23},\ldots,\theta_{2n})\in\mathcal{O}
\end{eqnarray}
where $\mathcal{O}$  is a compact manifold of dimension $2n-3$ with no boundary
which was introduced in the previous section. The elements in $\mathcal{O}$ represents the invariants of any solution to \eqref{eqqq}
generated by  rotations in the $(x_1 , x_2)$-plane, in the $(x_j , x_\alpha)$-plane, for any $j=1,2$ and $\alpha = 3, \ldots , n$. Here, with abuse of notation, we identify
$a = (a_1 , a_2) \in \R^2$, with $a= (a_1 , a_2 , 0) \in \R^n$.

For any set of parameters $A=(\lambda,\xi , a,\theta)\in \mathbb{R}^+\times \R^n \times \R^2 \times \R^{2n-3}$, we now introduce the function,
$$
Q_A (x)   =   \Theta_A (R_\theta^{-1}  (x - \xi) ),
$$
where $\Theta_A$ is defined in \eqref{thetaa}.
More explicitly,
\begin{eqnarray}\label{va1ok}
Q_A (x)
  =     \la^{-{n-2 \over 2}} \left| { x - \xi \over | x - \xi|} - R_\theta a {|x - \xi | \over \la } \right|^{2-n}   Q \left( {  {x - \xi \over \la } - R_\theta a |{x - \xi \over \la }|^2 \over 1-2 R_\theta a \cdot  {x - \xi \over \la }  + |a|^2 |{x - \xi \over \la }|^2 }\right).
\end{eqnarray}
Observe that $Q_A$ solves the equation in \eqref{eq1}, but it is far from satisfying the boundary conditions. For this reason, we correct $Q_A$, introducing its projection onto $H_0^1(\Omega_\ep)$.
Let us define $P_\ep Q_A$ to be
\begin{eqnarray}\label{proqa}
\left\{
  \begin{array}{ll}
-\Delta P_\ep Q_A=|Q_A|^{p-1}Q_A \quad  & {\rm in} \quad \Omega_\ep;\\
P_\ep Q_A=0 \quad  & {\rm on} \quad\partial \Omega_\ep.
  \end{array}
  \right.
\end{eqnarray}

We next give the description of the asymptotic behavior of the projection $P_\ep Q_A$ for $x\in\Omega_\ep$, as $\ep \to 0$. To do so, we need to introduce
Green's function $G(x,y)$ of the domain, namely $G$ satisfies
\begin{equation}\label{greenfun}
\left\{
  \begin{array}{ll}
-\Delta_x G(x,y) = \delta(x-y) \quad  &   \quad x\in\Omega;\\
G(x,y)=0 \quad  &   \quad x\in\partial \Omega,
  \end{array}
  \right.
\end{equation}
where $\delta (x)$ denotes the Dirac mass at
the origin, and its regular part  $H(x,y):= \Gamma (x-y) -G(x,y)$, where $\Gamma$ denotes the fundamental solution of the Laplacian,
\begin{equation}\label{fonsol}
\Gamma (x) = \gamma_n |x|^{2-n}.
\end{equation}
It is direct to see that
\begin{equation}\label{rugfun}
\left\{
  \begin{array}{ll}
-\Delta_x H(x,y) = 0 \quad  &   \quad x\in\Omega;\\
H(x,y) = \Gamma (x-y) \quad  &   \quad x\in\partial \Omega.
  \end{array}
  \right.
\end{equation}
Furthermore, we introduce the function
\begin{equation}\label{Fdef}
   F(\tau , a , \theta) :=
   \left\{ \begin{matrix}   Q \left( - {\tau \over |\tau|^2} - R_\theta a \right) |\tau |^{2-n}   & \hbox{ if } \tau \not= 0\, , \\
& \\     \lim_{z  \to \infty}  \, Q(z ) |z|^{n-2} & \hbox{ if } \tau=0.
 \end{matrix}
  \right.
 \end{equation}
 This is a smooth function in the set of parameters $\tau$, $a$ and $\theta$ satisfying
\eqref{txiaztea0}, \eqref{txiaztea1}, \eqref{ztea2n3}.
We have the validity of the following

\begin{lemma}
\label{lemma1} Let $\eta>0$ be fixed and assume that $A=(\lambda,\xi , a,\theta)\in \mathbb{R}^+\times \R^n \times \R^2\times \R^{2n-3}$ satisfies (\ref{txiaztea})-(\ref{ztea2n3}), with the additional assumption that $\xi = \la \tau$, and $\tau \not= 0$.  Let
\begin{align}
R(x)&:= P_\ve Q_A (x)- Q_A(x)+ \gamma_n^{-1} \la^{n-2 \over 2} Q(-R_\theta a)
H(x,\xi)  + {1\over \la^{n-2\over2} } F(\tau , a , \theta) {\ep^{n-2} \over |x|^{n-2}},
   \nonumber \end{align}
   where $F$ is defined in \eqref{Fdef}.
Then there exists a positive constant $c$ such that for any
$x\in\Omega\setminus B(0,  \ve ) $
\begin{align}
&\left | R(x)\right|\leq c\la^{\frac{n-2}{2}}\left[ {\frac{\ve^{n-2} (1+ \ve \la^{-n+1}
)}{|x|^{n-2}}}+\la^2 + {\ve^{n-2} \over \la^{n-2}}
\right],
\ \ \label{cru1}\\
&  \left|  \partial_{\la}R(x)\right|  \leq
c\la^{\frac{n-4}{2}}\left[ {\frac{\ve^{n-2} (1+ \ve \la^{-n+1}
)}{|x|^{n-2}}}+\la^2 + {\ve^{n-2} \over \la^{n-2}}
\right], \label{cru3}\\
&  \left|  \partial_{\tau_{i}}R(x)\right| \leq
c\la^{\frac{n}{2}}\left[ {\frac{\ve^{n-2} (1+ \ve \la^{-n}
)}{|x|^{n-2}}}+\la^2 + {\ve^{n-2} \over \la^{n-1}}
\right],\label{cru5}\\
&\left | \partial_{a_i} R(x)\right|\leq c\la^{\frac{n-2}{2}}\left[ {\frac{\ve^{n-2} (1+ \ve \la^{-n+1}
)}{|x|^{n-2}}}+\la^2 + {\ve^{n-2} \over \la^{n-2}}
\right],
\ \ \label{cru6}\\
&\left | \partial_{\theta_{ij}} R(x)\right|\leq c\la^{\frac{n-2}{2}}\left[ {\frac{\ve^{n-2} (1+ \ve \la^{-n+1}
)}{|x|^{n-2}}}+\la^2 + {\ve^{n-2} \over \la^{n-2}}
\right].
\ \ \label{cru7}
\end{align}
\end{lemma}

\begin{proof}
Let us introduce the scaled function $\hat R (y) = \la^{-{n-2 \over 2}}
 R( \ve y  )$, defined for $y \in \hat \Omega_\ve = \left( {\Omega  \over  \ve } \right)
\setminus  B(0,1)$.  Thus $-\Delta \hat R = 0$ in $\hat
\Omega_\ve $.
Furthermore,  $\hat \Omega_\ve \to \R^n\setminus
B(0,1)$ as $\ve \to 0$.
Observe that, if $z= {\ve y - \xi \over \la}$, then
$$
Q_A (\ve y ) = \la^{-{n-2 \over 2}} |z|^{2-n} Q \left( {z \over |z|^2} - R_\theta a \right).
$$
For any  $y \in \partial B(0,1)$ we have that
$$
\hat R (y) = -\la^{-{n-2 \over 2}}  Q_A (\ep y) + \gamma_n^{-1} Q(-
R_\theta a)  H(\ve y  , \xi ) +\la^{2-n} F(\tau , a ,\theta) .
$$
If $\xi = \la \tau$, and $\tau \not=0$, then $z= -\tau + {\ve \over \la } y$, and
a direct Taylor expansion gives that
\begin{align}
Q_A (\ve y)
&= \la^{-{n-2 \over 2}} Q \left(- {\tau \over |\tau|^2 } - R_\theta a   \right)  \, |\tau |^{2-n} \,  \left ( 1 + O({\ve \over \la} ) \right) \nonumber
\end{align}
uniformly for points $y \in \partial B(0,1)$. If $\tau = 0$, then
$$
Q_A (\ve y)
= \la^{-{n-2 \over 2}} \left( \lim_{w \to \infty} |z|^{2-n} Q( w) \right) \,  \left ( 1 + O({\ve \over \la} ) \right)
$$
Thus
we get the estimates
\begin{equation}\label{u1}
| \hat R (y) | \leq C (1+ {1\over \la^{n-2} } {\ve \over \la } )
\quad {\mbox {uniformly for}} \quad  y \in \partial B(0,1).
\end{equation}
Let us now take $y \in \partial \left( {\Omega \over  \ve } \right)$, and we have
$$
\hat R (y) = -\la^{-{n-2 \over 2}} Q_A (\ve y) + {\gamma_n \over | \ve y - \la \tau |^{n-2} } Q(R_\theta a )
+{1 \over \la^{n-2}}  F(\tau , a , \theta ){1 \over | y |^{n-2} }.
$$
Since  $|z|=| -\tau + {\ve y \over \la}| \to \infty$, as $\ve \to 0$, we get
$$
| \hat R (y) | \leq C (\la^2 +  ({\ve \over \la })^{n-2} ) \quad
{\mbox {uniformly for}}  y \in \partial \left( {\Omega \over  \ve } \right).
$$
A comparison argument for harmonic functions implies that
$$
|\hat R (y) | \leq C \left[ {1+ \ve \la^{1-n} \over |y|^{n-2} } +
\la^2  + ({\ve \over \la })^{n-2}  \right].
$$
This fact gives \equ{cru1}.

Let us now denote by $R_\la (x) = \partial_\la R(x)$ and define
$\hat R_\mu (y) = \la^{-{n-4 \over 2}} R( \ve y )$. A direct
computation shows that
$$
| \hat R_\la (y) | \leq C (1+ {1\over \la^{n-2} } {\ve \over \la } )
\quad  {\mbox {uniformly for}} \quad  y \in \partial B(0,1),
$$
and
$$
| \hat R_\la (y) | \leq C (\la^2 +  ({\ve \over \la })^{n-2} ) \quad
 {\mbox {uniformly for}} \quad  y \in \partial \left( {\Omega  \over  \ve } \right).
$$
This fact gives \equ{cru3}.

Finally, let $R_i (x) = \partial_{\tau_i} R(x)$ and $\hat R_i (y) =
\la^{-{n\over 2}}   R_i ( \ve y )$. We get the following
estimates
$$
| \hat R_i (y) | \leq C (1+ {\ve\over \la^{n} } ) \quad {\mbox {uniformly for}} y \in
\partial B(0,1),
$$
and
$$
| \hat R_i (y) | \leq C (\la^2 +  {\ve^{n-2} \over \la^{n-1} } )
\quad {\mbox {uniformly for}} y \in \partial \left( {\Omega  \over  \ve }
\right).
$$
This fact gives \equ{cru5}. In a similar way, one gets estimates \eqref{cru6} and \eqref{cru7}.
\phantom{pippo}
\end{proof}

\section{The expansion of the energy}\label{energys}

In this Section, we give the expansion of the energy function $J_\ep(P_\ep Q_A)$ which defined by
$$
J_\ep(u) = \frac{1}{2}\int_{\Omega_\ep}|\nabla u |^2dx-\frac{1}{p+1}
\int_{\Omega_\ep}|u |^{p+1}dx.
$$
We have the following result.

\begin{proposition}\label{energyes}
Let $\eta>0$ be fixed, and $A=(\lambda,\xi, a,\theta)\in \mathbb{R}^+\times \R^n \times \R^2\times \R^{2n-3}$ satisfies (\ref{txiaztea})-(\ref{ztea2n3}). Then
\begin{eqnarray}\label{energy0}
J_\ep(P_\ep Q_A) &=& c_1+\frac{1}{2}\left(\gamma_n^{-2} Q(-R_\theta a)^2 H(0,0) d^{n-2} + \frac{c_2}{d^{n-2}}  F(\tau , a , \theta) \right)\ep^{\frac{n-2}{2}} \nonumber \\
 &&+  \Pi (d ,  \tau , a , \theta ) \, \ve^{{n-1\over 2} }
\end{eqnarray}
and
\begin{eqnarray}\label{energy1}
\nabla_{(d, \tau , a , \theta ) }J_\ep(P_\ep Q_A) &=&  \nabla_{(d, \tau , a , \theta ) }\left[\frac{1}{2} \left(\gamma_n^{-2} Q(-R_\theta a)^2 H(0,0) d^{n-2} + \frac{c_2}{d^{n-2}}  F(\tau , a , \theta) \right)\right]\ep^{\frac{n-2}{2}} \nonumber \\
&&+ \Pi(d ,  \tau , a , \theta ) \, \ve^{{n-1\over 2} }
\end{eqnarray}
as $\ep\to0$, where  $\Pi$ denote a smooth function of its variables, which is uniformly bounded as $\ve \to 0$ for $(\lambda,\xi, a,\theta)$ satisfying  (\ref{txiaztea})-(\ref{ztea2n3}). Here $F$ is the function introduced in \eqref{Fdef}, and  $c_1$ and $c_2$ are the constants
\begin{eqnarray*}
c_1  =  \frac{1}{n}\int_{\R^n}   \left|Q \right|^{p+1} dz,\qquad
c_2  =   \int_{\mathbb{R}^n}    \left|Q \right|^{p} dz.
\end{eqnarray*}
\end{proposition}

\begin{proof}
We compute the energy
\begin{eqnarray}\label{energyg}
J_\ep(P_\ep Q_A)&=&\frac{1}{2}\int_{\Omega_\ep}|\nabla P_\ep Q_A |^2dx-\frac{1}{p+1}
\int_{\Omega_\ep}|P_\ep Q_A |^{p+1}dx.
\end{eqnarray}
Taking into the fact that $-\Delta P_\ep Q_A=|Q_A|^{p-1}Q_A$ in $\Omega_\ep$, and $P_\ep Q_A=0$ on
$\partial \Omega_\ep$, we  have
\begin{eqnarray}\label{energyq1}
 \int_{\Omega_\ep}|\nabla P_\ep Q_A |^2dx
&=& \int_{\Omega_\ep} |Q_A|^{p-1}Q_A P_\ep Q_A  dx \nonumber\\
&=&\int_{\Omega_\ep} |Q_A|^{p+1}   dx+\int_{\Omega_\ep} |Q_A|^{p-1}Q_A (P_\ep Q_A -Q_A) dx.
\end{eqnarray}
Moreover, by a Taylor expansion, for some $t\in(0,1)$,
\begin{eqnarray}\label{energyq2}
\int_{\Omega_\ep}|P_\ep Q_A |^{p+1}dx
&=&\int_{\Omega_\ep}|Q_A+(P_\ep Q_A-Q_A) |^{p+1}dx\nonumber\\
&=& \int_{\Omega_\ep} |Q_A|^{p+1}dx+(p+1)\int_{\Omega_\ep} |Q_A|^{p-1}Q_A (P_\ep Q_A -Q_A) dx\nonumber\\
&&+\frac{p(p+1)}{2}\int_{\Omega_\ep}\left(tP_\ep Q_A+(1-t)Q_A \right)^{p-1}(P_\ep Q_A -Q_A)^2dx,
\end{eqnarray}
From (\ref{energyg}), (\ref{energyq1}) and (\ref{energyq2}), we get
\begin{eqnarray}\label{energyrt}
J_\ep(P_\ep Q_A)
&=&\frac{1}{n}\int_{\Omega_\ep} |Q_A|^{p+1}dx-\frac{1}{2}\int_{\Omega_\ep} |Q_A|^{p-1}Q_A (P_\ep Q_A -Q_A) dx\nonumber\\
&&-\frac{p}{2}\int_{\Omega_\ep}\left(tP_\ep Q_A+(1-t)Q_A \right)^{p-1}(P_\ep Q_A -Q_A)^2dx.
\end{eqnarray}
We will estimate each term in the following, and then the result in Proposition is a consequence of the following Lemmas.
\end{proof}

\begin{lemma}\label{eng1} Let $\eta>0$ be fixed, and $A=(\lambda,\xi, a,\theta)\in \mathbb{R}^+\times \R^n \times \R^2\times \R^{2n-3}$ satisfies (\ref{txiaztea})-(\ref{ztea2n3}).
It holds
\begin{eqnarray}\label{en1}
\int_{\Omega_\ep} |Q_A|^{p+1}dx &=&  \int_{\R^n}   \left|Q \right|^{p+1} dz+\ve^{n\over 2} \Pi (d, \tau , a, \theta ),
\end{eqnarray}
where $\Pi$ is a smooth function of its variables, which is uniformly bounded as $\ve \to 0$ for $(\lambda,\xi, a,\theta)$ satisfying  (\ref{txiaztea})-(\ref{ztea2n3}).
\end{lemma}

\begin{proof}
We decompose it as
$$
\int_{\Omega_\ve } |Q_A |^{p+1}   dx = \int_{\ve <|x| < \delta} |Q_A |^{p+1} dx + \int_{|x|> \delta} |Q_A |^{p+1} dx
$$
for some $\delta >0$ fixed and small. In the region $ \ve <|x| < \delta$ we introduce the change of variables $ y = {x- \xi \over \la } = {x\over \la} -\tau$, so that
\begin{eqnarray*}
 \int_{\ve < |x | < \delta} |Q_A |^{p+1} \, dx & = &
\int_{B_{1,\ve}} \left[ \left| {y \over |y|} -  R_\theta a |y| \right|^{2-n}
|Q| \left( {{y\over |y|^2} -  R_\theta a \over | {y \over |y|^2 } - R_\theta a |^2 } \right) \right]^{p+1} dy,
\end{eqnarray*}
where $B_{1,\ve} := B (-\tau , {\delta \over \la} ) \setminus B (-\tau , {\ve \over \la} )$.
Since $|w|^{2-n} Q ({w\over |w|^2} ) = Q(w)$, and using the change of variables $z= {y\over |y|^2}$, we then have
\begin{eqnarray*}
 \int_{\ve < |x | < \delta} |Q_A |^{p+1} \, dx & = &
  \int_{B_{1,\ve} } \left[ |y|^{2-n}
|Q| \left( {y\over |y|^2} -  R_\theta a \right) \right]^{p+1} dy
\\
& =&  \int_{B_{1,\ve}}  |y|^{2n} \left[
|Q| \left({y\over |y|^2} -  R_\theta a  \right) \right]^{p+1} dy \\
& =&  \int_{B_{2,\ve}}   \left| Q \left( z - R_\theta a \right) \right|^{p+1} dz =  \int_{\R^n}   \left|Q \right|^{p+1} dz  + \la^n O(1),
\end{eqnarray*}
where $B_{2,\ve} := B({-\tau \over |\tau |^2} , {\la \over \ve} ) \setminus B({-\tau \over |\tau |^2} , {\la \over \delta} )$, if $\tau \not= 0$,
and $B_{2,\ve} := B(0, {\la \over \ve} ) \setminus B(0, {\la \over \delta} )$ if $\tau = 0$. Thus, the above estimate holds true for a generic function $O(1)$ of the parameters $ (d, \tau , a, \theta )$, which is uniformly bounded as $\ve \to 0$.
On the other hand,  in the set ${|x|> \delta}$ we have that
$
\left|  \int_{|x |> \delta} |Q_A |^{p+1} \, dx \right|  \leq C \la^n.
$
This concludes the proof of the Lemma.
\end{proof}

\begin{lemma}\label{eng2}
Let $\eta>0$ be fixed, and $A=(\lambda,\xi, a,\theta)\in \mathbb{R}^+\times \R^n \times \R^2\times \R^{2n-3}$ satisfies (\ref{txiaztea})-(\ref{ztea2n3}).
It holds, as $\ve \to 0$,
\begin{eqnarray}\label{enw}
\int_{\Omega_\ep} & & |Q_A|^{p-1}Q_A (P_\ep Q_A -Q_A) dx =  -\gamma_n^{-2} Q(-R_\theta a)^2 H(0,0)\la^{n-2}\nonumber\\
&&-\left(\frac{\ep}{\la}\right)^{n-2}  F(\tau , a , \theta) \int_{\mathbb{R}^n}    \left|Q \right|^{p} dz  + \Pi (d, \tau , a, \theta )\, \ve^{{n-1\over 2} },
\end{eqnarray}
where  $\Pi$ is a smooth function of its variables, which is uniformly bounded as $\ve \to 0$ for $(\lambda,\xi, a,\theta)$ satisfying  (\ref{txiaztea})-(\ref{ztea2n3}). Here $F$ is the function introduced in \eqref{Fdef}.
\end{lemma}

\begin{proof}
By Lemma \ref{lemma1}, we have
\begin{eqnarray}\label{energesw2}
&&\int_{\Omega_\ep} |Q_A|^{p-1}Q_A (P_\ep Q_A -Q_A) dx
=-\gamma_n^{-1}\la^{\frac{n-2}{2}}Q(-R_\theta a)\int_{\Omega_\ep} |Q_A|^{p-1}Q_A H(x,\xi)dx\nonumber\\
&&-\frac{\ep^{n-2}}{\la^{\frac{n-2}{2}}} F (\tau , a , \theta ) \int_{\Omega_\ep} |Q_A|^{p-1}Q_A \frac{dx}{|x|^{n-2}}+\int_{\Omega_\ep} |Q_A|^{p-1}Q_AR(x)dx\nonumber\\
&:=&I_1+I_2+I_3.
\end{eqnarray}
To estimate $I_1$, we write
$$
\int_{\Omega_\ep} |Q_A|^{p-1} Q_A H(x,\xi )  dx = \int_{\ep<|x |<\delta } |Q_A|^{p-1} Q_A H(x,\xi) dx + \int_{\Omega \cap |x|>\delta  } |Q_A|^{p-1} Q_A H(x,\xi)dx,
$$
for some positive, small and fixed $\delta$. In the first region, the function $H$ is smooth, and in particular it has bounded derivatives. Thus, by Taylor expansions, we get
\begin{eqnarray*}
&& \int_{\ep<|x |<\delta } |Q_A|^{p-1} Q_A H(x,0)  dx =
  H(0,0)\left(  \int_{\ep<|x |<\delta } |Q_A|^{p-1} Q_A  dx  \right)\\
&&+ O( \int_{\ep<|x|<\delta } |Q_A|^{p-1} Q_A   |x |dx ) + \la O ( \int_{\ep<|x|<\delta } |Q_A|^{p-1} Q_A   dx ).
\end{eqnarray*}
Taking the change of variables $ y = {x -\xi \over \la}$,  using the invariance of $Q$ under Kelvin transform, and using the change of variables $z= {y\over |y|^2}$, we get
\begin{eqnarray*}
 \int_{\ep<|x | < \delta} |Q_A |^{p-1} Q_A \, dx
&=& \la^{n-2 \over 2} \int_{B_{1,\ep}} \left[ |y|^{2-n}
|Q| \left( {y\over |y|^2} - R_\theta a \right) \right]^{p} dy
\\
& =&  \la^{n-2 \over 2}\int_{B_{1,\ep}} |y|^{-(n+2)} \left[
|Q| \left({y\over |y|^2} - R_\theta a  \right) \right]^{p}  dy \\
& =& \la^{n-2 \over 2} \int_{B_{2,\ep}} {1\over |z|^{n-2}}  \left| Q \left( z - R_\theta a \right) \right|^{p} dz  \\
&=& \la^{n-2 \over 2} \int_{\mathbb{R}^n}   {1\over |z+  R_\theta  a |^{n-2} } \left|Q \right|^{p} dz  + O(  \la^{n+2 \over 2}   ),
\end{eqnarray*}
as $\ve \to 0$. Here again $B_{1,\ve} := B (-\tau , {\delta \over \la} ) \setminus B (-\tau , {\ve \over \la} )$, while $B_{2,\ve} := B({-\tau \over |\tau |^2} , {\la \over \ve} ) \setminus B({-\tau \over |\tau |^2} , {\la \over \delta} )$, if $\tau \not= 0$,
and $B_{2,\ve} := B(0, {\la \over \ve} ) \setminus B(0, {\la \over \delta} )$ if $\tau = 0$.
Recall now that
$$
 Q(-R_\theta a ) = \gamma_n  \int_{\R^n}   {1\over |z+ R_\theta a |^{n-2} } \left|Q \right|^{p} dz
$$
Thus we get
\begin{equation}
\label{inter}
 \int_{\ep<|x | < \delta} |Q_A |^{p-1} Q_A   dx = \gamma_n^{-1} \la^{n-2\over 2} Q (-R_\theta a ) + O (\la^{n+2 \over 2} ).
\end{equation}
On the other hand, using again the change of variables $ y = {x-\xi   \over \la }$
one finds directly that
$$
 \int_{|x  |<\delta } |Q_A|^{p-1} Q_A \,  |x  |dx = O (\la^{n\over 2} ).
$$
On the other hand, we observe that, in the region where $|x| > \delta$, one has
\begin{equation}
\label{caldi}
|Q_A (x) | \leq C \la^{n-2 \over 2}, \quad {\mbox {for all}} \quad x \in \Omega, \quad |x| >\delta,
\end{equation}
where the constant $C$ is independent of $\ve$. Indeed, to prove \eqref{caldi}, we start with the observation that, in the region under consideration, one has
$$
|Q_A (x) | \leq c |Q_{\bar A} (x)|, \quad {\mbox {where}} \quad \bar A = (\la , 0 , a , \theta),
$$
for some constant $c$, independent of $\ve$. Now, if we set $y={x\over \la}$, we have, in the region under consideration
\begin{eqnarray*}
|Q_A (x) | &\leq & C \la^{-{n-2 \over 2} } \left| {y \over |y|^2} - R_\theta a \right|^{2-n} \, |y|^{2-n} \,
Q\left( { {y\over |y|^2} - R_\theta a \over \left| {y \over |y|^2 } - R_\theta a \right|^2 } \right) \\
& = & C \la^{-{n-2 \over 2} } \, |y|^{2-n} \, Q\left(  {y\over |y|^2} - R_\theta a  \right) \leq C \la^{n-2 \over 2}.
\end{eqnarray*}
Thus the validity of \eqref{caldi} follows. A direct consequence of \eqref{caldi} is
$$
 \int_{\Omega \cap |x |>\delta  } |Q_A|^{p-1} Q_A H(x,0) \, dx = O(\la^{n+2 \over 2} ).
$$
Then, we find
\begin{eqnarray}\label{energesw2i1}
I_1&=&-\gamma_n^{-2}\la^{n-2}Q(-R_\theta a)^2H(0,0) + O (\la^{n - 1 }).
\end{eqnarray}
Let us now estimate $I_2$.  We split the integral as follows
$$
\int_{\Omega_\ep} |Q_A|^{p-1}Q_A \frac{1}{|x|^{n-2}}dx
= \int_{\ep<|x|<\delta} |Q_A|^{p-1}Q_A \frac{1}{|x|^{n-2}}dx +\int_{\Omega\cap |x|>\delta} |Q_A|^{p-1}Q_A \frac{1}{|x|^{n-2}}dx.
$$
Using again \eqref{caldi}, we see that $ \int_{\Omega\cap |x|>\delta} |Q_A|^{p-1}Q_A \frac{1}{|x|^{n-2}}dx = O(\la^{n+2 \over 2} ) $.
Using the invariance of $Q$ under Kelvin transform, and using the changes of variables, first $y = {x-\xi \over \la}$ and then $z= {y\over |y|^2}$, we get
\begin{eqnarray}\label{energesw2b}
\int_{\Omega_\ep \cap |x| < \delta } |Q_A|^{p-1}Q_A \frac{1}{|x|^{n-2}}dx
&=&\la^{-\frac{n-2}{ 2}} \int_{B_{1,\ve} } \left[ |y|^{2-n}
|Q| \left( {y\over |y|^2} - R_\theta a \right) \right]^{p}\frac{1}{|y|^{n-2}} dy
\nonumber\\
& =&  \la^{-\frac{n-2}{ 2}} \int_{B_{1,\ve} }  |y|^{-2n} \left[
|Q| \left({y\over |y|^2} - R_\theta a  \right) \right]^{p} dy\nonumber\\
&=&  \la^{-\frac{n-2}{ 2}}  \int_{B_{2,\ve} }  \left| Q \left( z - R_\theta a \right) \right|^{p} dz
= \la^{-\frac{n-2}{ 2}} \left( \int_{\mathbb{R}^n}    \left|Q \right|^{p} dz  + O(\la^2 ) \right) .
\end{eqnarray}
Then
\begin{eqnarray}\label{energesw2i1}
I_2 = -\frac{\ep^{n-2}}{\la^{n-2}}\left(  F(\tau , a , \theta) \int_{\mathbb{R}^n}    \left|Q \right|^{p} dz  +  O (\la^2 )\right).
\end{eqnarray}
We conclude with the estimate for $I_3$. We use the result in Lemma \ref{lemma1}, and in particular estimate \eqref{cru1}, to get
\begin{eqnarray}\label{energesw2i33}
|I_3|= |\int_{\Omega_\ep} |Q_A|^{p-1}Q_AR (x)dx|
\leq  c\ep^{\frac{1}{2}} ( |I_1| + |I_2| ),
\end{eqnarray}
for some positive constant $c$.
This concludes the proof of  Lemma.
\end{proof}

\begin{lemma}\label{eng3}
Under the same assumptions as in Proposition \ref{energyes}, it holds
\begin{eqnarray}\label{en3}
\int_{\Omega_\ep}\left(tP_\ep Q_A+(1-t)Q_A \right)^{p-1}(P_\ep Q_A -Q_A)^2dx =  \Pi (d, \tau , a, \theta ) \, \ve^{{n-1\over 2} }.
\end{eqnarray}
\end{lemma}

\begin{proof}
Using (\ref{energyrt}), we have
\begin{eqnarray}\label{energyrtdisan}
&& |\int_{\Omega_\ep}\left(tP_\ep Q_A+(1-t)Q_A \right)^{p-1}(P_\ep Q_A -Q_A)^2dx|\nonumber\\
&\leq &c\int_{\Omega_\ep} |Q_A|^{p-1}\left(\la^{ n-2}+\frac{\ep^{2(n-2)}}{\la^{ n-2 }}\frac{1}{|x|^{2(n-2)}}\right)dx.
\end{eqnarray}
Since
\begin{eqnarray}\label{energyrtdisan1}
 \la^{ n-2}\int_{\Omega_\ep} |Q_A|^{p-1} dx&=&\la^{ n-2}\int_{\ep<|x|<\delta} |Q_A|^{p-1} dx+\la^{ n-2}\int_{\Omega\cap |x|>\delta} |Q_A|^{p-1} dx\nonumber\\
&=&\la^{2(n-2)} \int_{\frac{\ep}{\la}<|y|<{\delta \over \la}} \left[ |y|^{2-n}
|Q| \left( {y\over |y|^2} - R_\theta a \right) \right]^{p-1} dy+O(\la^{ n+2})
\nonumber\\
& =&  \la^{2(n-2)}\int_{\frac{\ep}{\la}<|y|<{\delta \over \la}}  |y|^{-4} \left[
|Q| \left({y\over |y|^2} - R_\theta a  \right) \right]^{p-1}  dy+O(\la^{ n+2}) \nonumber\\
& =& \la^{2(n-2) }  \int_{\frac{\la}{\delta}<|z|<{\la \over \ep}} {1\over |z|^{2(n-2)}}  \left| Q \left( z - R_\theta a \right) \right|^{p-1} dz +O(\la^{ n+2}) \nonumber\\
&=& \la^{2(n-2)} (\int_{\mathbb{R}^n}   {1\over |z+  R_\theta  a |^{2(n-2)} } \left|Q \right|^{p-1} dz  + o(  1  ))+O(\la^{ n+2})\nonumber\\
&=&O(\la^{ n+2}),
\end{eqnarray}
and
\begin{eqnarray}\label{energyrtdisan2}
&&\frac{\ep^{2(n-2)}}{\la^{ n-2 }}\int_{\Omega_\ep} |Q_A|^{p-1}\frac{1}{|x|^{2(n-2)}} dx\nonumber\\
&=&\frac{\ep^{2(n-2)}}{\la^{ n-2 }}\int_{\ep<|x|<\delta} |Q_A|^{p-1}\frac{1}{|x|^{2(n-2)}}dx+\frac{\ep^{2(n-2)}}{\la^{ n-2 }}\int_{\Omega\cap |x|>\delta} |Q_A|^{p-1}\frac{1}{|x|^{2(n-2)}} dx\nonumber\\
&=&\frac{\ep^{2(n-2)}}{\la^{2( n-2) }} \int_{\frac{\ep}{\la}<|y|<{\delta \over \la}} \left[ |y|^{2-n}
|Q| \left( {y\over |y|^2} - R_\theta a \right) \right]^{p-1}\frac{1}{|y|^{2(n-2)}} dy+O(\la^{4})\frac{\ep^{2(n-2)}}{\la^{ n-2 }}
\nonumber\\
& =& \frac{\ep^{2(n-2)}}{\la^{2( n-2) }}\int_{\frac{\ep}{\la}<|y|<{\delta \over \la}}  |y|^{-3(n-2)} \left[
|Q| \left({y\over |y|^2} - R_\theta a  \right) \right]^{p-1}  dy+O(\la^{4})\frac{\ep^{2(n-2)}}{\la^{ n-2 }} \nonumber\\
& =&\frac{\ep^{2(n-2)}}{\la^{2( n-2) }}  \int_{\frac{\la}{\delta}<|z|<{\la \over \ep}} {1\over |z|^{n-6}}  \left| Q \left( z - R_\theta a \right) \right|^{p-1} dz +O(\la^{4})\frac{\ep^{2(n-2)}}{\la^{ n-2 }} \nonumber\\
&=&\frac{\ep^{2(n-2)}}{\la^{2( n-2) }} (\int_{\mathbb{R}^n}   {1\over |z+  R_\theta  a |^{n-6} } \left|Q \right|^{p-1} dz  + o(  1  ))+O(\la^{4})\frac{\ep^{2(n-2)}}{\la^{ n-2 }}\nonumber\\
&=&O(\frac{\ep^{2(n-2)}}{\la^{2( n-2) }}).
\end{eqnarray}
Then (\ref{en3}) follows from (\ref{energyrtdisan}) to (\ref{energyrtdisan2}).
\end{proof}

\medskip
We conclude this section with the proof of \eqref{energy1}. More precisely, we prove
\begin{align}\label{energy11}
\partial_d J_\ep(P_\ep Q_A)  =&  \partial_d\left[\frac{1}{2} \left(\gamma_n^{-2} Q(-R_\theta a)^2 H(0,0) d^{n-2} + \frac{c_2}{d^{n-2}}  F(\tau , a , \theta)  \right)\right]\ep^{\frac{n-2}{2}}\nonumber \\
 &+\Pi (d, \tau , a, \theta )\, \ve^{{n-1\over 2} },
\end{align}
as $\ep\to0$, where  $\Pi$ is a smooth function of the  variables $(d, \tau , a, \theta )$, which is uniformly bounded as $\ve \to 0$ for $(\lambda,\xi, a,\theta)$ satisfying  (\ref{txiaztea})-(\ref{ztea2n3}).
The estimates for the other derivatives can be obtain in a similar way.

{\bf Proof of \eqref{energy11}}:

We have
\begin{align*}
\partial_dJ_\ep(P_\ep Q_A) =&\partial_d\left(\frac{1}{2}\int_{\Omega_\ep}|\nabla P_\ep Q_A |^2dx-\frac{1}{p+1}
\int_{\Omega_\ep}|P_\ep Q_A |^{p+1}dx\right)\nonumber\\
= & \int_{\Omega_\ep} \nabla P_\ep Q_A \nabla (\partial_d(P_\ep Q_A))dx-
\int_{\Omega_\ep}|P_\ep Q_A |^{p}\partial_d(P_\ep Q_A)dx
\end{align*}
Since the function $P_\ep Q_A$ satisfies (\ref{proqa}), we find
\begin{align*}
\partial_dJ_\ep(P_\ep Q_A) = -  \int_{\Omega_\ep} \left[ |P_\ep Q_A |^{p}-|Q_A|^{p-1}Q_A\right]\partial_d(P_\ep Q_A)dx.
\end{align*}
By a Taylor expansion, for some $t\in(0,1)$,
\begin{align*}
|P_\ep Q_A |^{p}
 =& |Q_A+(P_\ep Q_A-Q_A) |^{p}=|Q_A|^{p-1}Q_A +p |Q_A|^{p-2}Q_A (P_\ep Q_A -Q_A)  \nonumber\\
& +\frac{p(p-1)}{2} \left(tP_\ep Q_A+(1-t)Q_A \right)^{p-2}(P_\ep Q_A -Q_A)^2.
\end{align*}
Then
\begin{align}\label{esc1a}
\partial_dJ_\ep(P_\ep Q_A)
=&-  \int_{\Omega_\ep} \left[ p |Q_A|^{p-2}Q_A (P_\ep Q_A -Q_A)\right]\partial_d\big(Q_A+(P_\ep Q_A-Q_A)\big)dx\nonumber\\
&-  \frac{p(p-1)}{2} \int_{\Omega_\ep} \left(tP_\ep Q_A+(1-t)Q_A \right)^{p-2}(P_\ep Q_A -Q_A)^2\partial_d\big(Q_A+(P_\ep Q_A-Q_A)\big)dx\nonumber\\
=&-  \int_{\Omega_\ep} \left[ p |Q_A|^{p-2}Q_A (P_\ep Q_A -Q_A)\right]\partial_d Q_A dx\nonumber\\
&-  \int_{\Omega_\ep} \left[ p |Q_A|^{p-2}Q_A (P_\ep Q_A -Q_A)\right]\partial_d\big( P_\ep Q_A-Q_A \big)dx\nonumber\\
&-  \frac{p(p-1)}{2} \int_{\Omega_\ep} \left(tP_\ep Q_A+(1-t)Q_A \right)^{p-2}(P_\ep Q_A -Q_A)^2\partial_d\big(Q_A+(P_\ep Q_A-Q_A)\big)dx\nonumber\\
=&-  \int_{\Omega_\ep} \left[ p |Q_A|^{p-2}Q_A (P_\ep Q_A -Q_A)\right]\partial_d Q_A dx\nonumber\\
&+O\left(\int_{\Omega_\ep}   |Q_A|^{p-2}Q_A (P_\ep Q_A -Q_A)^2  dx\right)\nonumber\\
=&
 - \int_{\Omega_\ep} \partial_d\left[   |Q_A|^{p }  \right] (P_\ep Q_A -Q_A) dx+O\left(\int_{\Omega_\ep}   |Q_A|^{p-2}Q_A (P_\ep Q_A -Q_A)^2  dx\right)\nonumber\\
  =& \partial_d\left(-\int_{\Omega_\ep}  |Q_A|^{p } (P_\ep Q_A -Q_A) dx \right)
+\int_{\Omega_\ep} |Q_A|^{p } \partial_d (P_\ep Q_A -Q_A)    dx\nonumber\\
&+O\left(\int_{\Omega_\ep}   |Q_A|^{p-2}Q_A (P_\ep Q_A -Q_A)^2  dx\right).
\end{align}
From Lemma \ref{eng2}, we have that
\begin{align}\label{esc1b}
 & \partial_d\left(-\int_{\Omega_\ep}  |Q_A|^{p } (P_\ep Q_A -Q_A) dx \right)\nonumber\\
 =&   \partial_d \left(\gamma_n^{-2} Q(-R_\theta a)^2 H(0,0) d^{n-2} + \frac{c_2}{d^{n-2}}  F(\tau , a , \theta)  \right)\ep^{\frac{n-2}{2}} +\Pi (d, \tau , a, \theta )\, \ve^{{n-1\over 2} }\nonumber\\
 =&   (n-2)d^{-1} \left(\gamma_n^{-2} Q(-R_\theta a)^2 H(0,0) d^{n-2} - \frac{c_2}{d^{n-2}}  F(\tau , a , \theta)  \right)\ep^{\frac{n-2}{2}} +\Pi (d, \tau , a, \theta )\, \ve^{{n-1\over 2} }
\end{align}
Moreover, for the second term in (\ref{esc1a}), by Lemma \ref{lemma1}, we have
\begin{align*}
& \int_{\Omega_\ep} |Q_A|^{p } \partial_d (P_\ep Q_A -Q_A)    dx\nonumber\\
=& \int_{\Omega_\ep} |Q_A|^{p } \partial_d \left[-\gamma_n^{-1}\la^{\frac{n-2}{2}}Q(-R_\theta a)H(x,\xi)-\frac{\ep^{n-2}}{\la^{\frac{n-2}{2}}} F (\tau , a , \theta )\frac{1}{|x|^{n-2}}+R(x)\right]    dx\nonumber\\
=& \int_{\Omega_\ep} |Q_A|^{p } \partial_\lambda \left[-\gamma_n^{-1}\la^{\frac{n-2}{2}}Q(-R_\theta a)H(x,\xi)-\frac{\ep^{n-2}}{\la^{\frac{n-2}{2}}} F (\tau , a , \theta )\frac{1}{|x|^{n-2}}+R(x)\right]\frac{\partial \lambda}{\partial d}    dx\nonumber\\
=&\frac{n-2}{2} \ep^{\frac{1}{2}}\int_{\Omega_\ep} |Q_A|^{p }  \left[-\gamma_n^{-1}\la^{\frac{n-2}{2}-1}Q(-R_\theta a)H(x,\xi)+\frac{\ep^{n-2}}{\la^{\frac{n-2}{2}+1}} F (\tau , a , \theta )\frac{1}{|x|^{n-2}}\right]dx\nonumber\\
&+\ep^{\frac{1}{2}}\int_{\Omega_\ep} |Q_A|^{p }\partial_\lambda R(x) dx\nonumber\\
&\quad \mbox{since}\ \ \lambda=\sqrt{\ep} d\nonumber\\
=&\frac{n-2}{2} d^{-1}\int_{\Omega_\ep} |Q_A|^{p }  \left[-\gamma_n^{-1}\la^{\frac{n-2}{2} }Q(-R_\theta a)H(x,\xi)+\frac{\ep^{n-2}}{\la^{\frac{n-2}{2} }} F (\tau , a , \theta )\frac{1}{|x|^{n-2}}\right]dx\nonumber\\
&+\ep^{\frac{1}{2}}\int_{\Omega_\ep} |Q_A|^{p }\partial_\lambda R(x) dx\nonumber\\
 =&\frac{n-2}{2} d^{-1}\left[I_1-I_2\right] +\ep^{\frac{1}{2}}\int_{\Omega_\ep} |Q_A|^{p }\partial_\lambda R(x) dx.
\end{align*}
where $I_1$ and $I_2$ are defined in (\ref{energesw2}), with
$$
I_1-I_2=-\left(\gamma_n^{-2} Q(-R_\theta a)^2 H(0,0) d^{n-2} - \frac{c_2}{d^{n-2}}  F(\tau , a , \theta)  \right)\ep^{\frac{n-2}{2}}  + \Pi (d, \tau , a, \theta )\, \ve^{{n-1\over 2} },
$$
and  from Lemma \ref{lemma1}, we have $|\partial_\lambda R(x)|\leq c\lambda^{-1}  |R(x)|$, then by (\ref{energesw2i33}), we get
\begin{align*}
|\ep^{\frac{1}{2}}\int_{\Omega_\ep} |Q_A|^{p }\partial_\lambda R(x) dx|\leq& c \lambda^{-1}
\ep^{\frac{1}{2}}\int_{\Omega_\ep} |Q_A|^{p }  R(x) dx=c d\int_{\Omega_\ep} |Q_A|^{p }  R(x) dx\nonumber\\
\leq & c\ep^{\frac{1}{2}} ( |I_1| + |I_2| ).
\end{align*}
Thus
\begin{align}\label{energesw2jk}
& \int_{\Omega_\ep} |Q_A|^{p } \partial_d (P_\ep Q_A -Q_A)    dx\nonumber\\
=&-\frac{n-2}{2} d^{-1}\left(\gamma_n^{-2} Q(-R_\theta a)^2 H(0,0) d^{n-2} - \frac{c_2}{d^{n-2}}  F(\tau , a , \theta)  \right)\ep^{\frac{n-2}{2}}  + \Pi (d, \tau , a, \theta )\, \ve^{{n-1\over 2} }.
\end{align}
Lastly, using Lemma \ref{lemma1}, as a computation in Lemma \ref{eng2}, we have
\begin{align}\label{energesw2jkfg}
O\left(\int_{\Omega_\ep}   |Q_A|^{p-2}Q_A (P_\ep Q_A -Q_A)^2  dx\right)=\Pi (d, \tau , a, \theta )\, \ve^{{n-1\over 2} },
\end{align}
as $\ep\to0$, where  $\Pi$ is a smooth function of the  variables $(d, \tau , a, \theta )$, which is uniformly bounded as $\ve \to 0$ for $(\lambda,\xi, a,\theta)$ satisfying  (\ref{txiaztea})-(\ref{ztea2n3}).

Therefore, by (\ref{esc1a}), (\ref{esc1b}), (\ref{energesw2jk}) and (\ref{energesw2jkfg}), we obtain
\begin{align*}
\partial_dJ_\ep(P_\ep Q_A)= &\frac{n-2}{2} d^{-1}\left(\gamma_n^{-2} Q(-R_\theta a)^2 H(0,0) d^{n-2} - \frac{c_2}{d^{n-2}}  F(\tau , a , \theta) \right)\ep^{\frac{n-2}{2}}  + \Pi (d, \tau , a, \theta )\, \ve^{{n-1\over 2} }\nonumber\\
=& \partial_d\left[\frac{1}{2}\left(\gamma_n^{-2} Q(-R_\theta a)^2 H(0,0) d^{n-2}+ \frac{c_2}{d^{n-2}}  F(\tau , a , \theta)  \right)\ep^{\frac{n-2}{2}}\right]  + \Pi (d, \tau , a, \theta )\, \ve^{{n-1\over 2} }.
\end{align*}
That is, \eqref{energy11} holds.

\section{Scheme of the proof}\label{exst}

By the change of variable,
\begin{eqnarray}\label{defv}
v(y)=\ep^{\frac{1}{p-1}}u(\sqrt{\ep} y).
\end{eqnarray}
Problem (\ref{eq1}) has a solution $u$ if and only if $v$ solves the following problem
\begin{eqnarray}\label{eq1tov}
\left\{
  \begin{array}{ll}
\Delta v+|v|^{p-1}v=0, \quad  & {\rm in} \  D_\ep;\\
v=0, \quad  & {\rm on} \ \partial D_\ep,
  \end{array}
  \right.
\end{eqnarray}
where $D_\ep:=\frac{\Omega_\ep}{\sqrt{\ep}}=\frac{\Omega}{\sqrt{\ep}}\backslash B(0,\sqrt{\ep})$.

In expanded variable, the solution that we are looking for looks like
\begin{eqnarray}\label{defvfir}
v(y)=V(y)+\phi(y),\qquad \mbox{where}\ \ V(y)=\ep^{\frac{1}{p-1}}P_\ep Q_A(\sqrt{\ep} y),
\end{eqnarray}
where $P_\ep Q_A$ is defined in (\ref{proqa}). We observe that the function $V$ is nothing but the projection onto $H_0^1(D_\ep)$ of the function $\ep^{\frac{1}{p-1}} Q_A(\sqrt{\ep} y)$. We also observe that, if $A = (\la , \xi , a , \theta)$, then
$$\ep^{\frac{1}{p-1}} Q_A(\sqrt{\ep} y)\equiv Q_{\tilde{A}}(y), \quad {\mbox {with}} \quad \tilde{A}=(d,d \tau , a,\theta)
 $$
since $\la = d \sqrt{\ve}$ and $\xi = \la \tau$,
where $Q_A$ is given in (\ref{va1ok}).

Rewriting the result contained in Lemma \ref{lemma1}, we see that
as $\ep \to 0$,
\begin{eqnarray} \label{defv0}
V (y)&=& Q_{\tilde A} (y) +  \ve^{n-2 \over 2} (1+ {1\over |y|^{n-2}} ) \Xi_{\tilde A} (y),
\end{eqnarray}
uniformly on compact sets of $D_\ve$. Here $\Xi_{\tilde A} (y)$ is a smooth function, which is uniformly bounded for $y \in D_\ve$, as $\ve \to 0$, and for sets of parameters $\tilde A$ satisfying (\ref{txiaztea})-(\ref{ztea2n3}).

In terms of $\phi$, problem (\ref{eq1tov}) becomes
\begin{equation}\label{eq1tova1}
L(\phi)=-N(\phi)-E, \quad   {\rm in} \  D_\ep, \quad
\phi=0, \quad   {\rm on} \ \partial D_\ep,
\end{equation}
where
\begin{equation}\label{eqv1}
L(\phi)=\Delta\phi+p V^{p-1}\phi, \quad
N(\phi)= (V+\phi)^p-V^p-pV^{p-1}\phi,
\end{equation}
and
\begin{eqnarray}\label{eqv3}
E&=&  V^p-|Q_{\tilde{A}}|^{p-1}Q_{\tilde{A}}, \quad {\mbox {with}} \quad \tilde{A}=(d,d \tau , a,\theta).
\end{eqnarray}
Consider the following functions, for any $j=0,1,2,\ldots , 3n-1$,
\begin{eqnarray} \label{defv0}
Z_j(y)=  \ep^{\frac{1}{p-1}} \widetilde{Z}_j (\sqrt{\ep} y),\quad y\in D_\ep, \quad \widetilde{Z}_j (x) =
\Theta_A [z_j] (x),
\end{eqnarray}
where $\Theta_A$ is the operator defined in \eqref{thetaa0}. Observe that
\begin{equation}\label{neve1}
Z_j (y) = \Theta_{\tilde A} [z_j] (y).
\end{equation}
In order to solve problem (\ref{eq1tova1}), we first consider the linear problem.
 Let $\eta>0$ be fixed as in (\ref{txiaztea}), and assume that the set of parameters $A=(\lambda,\xi, a,\theta)\in \mathbb{R}^+\times \R^n \times \R^2\times \R^{2n-3}$ satisfies (\ref{txiaztea})-(\ref{ztea2n3}). Given a function $h$, we consider the problem of finding a function $\phi$ and real numbers $c_j$, $j=0,1,2,\ldots,3n-1$ such that
\begin{eqnarray}\label{lineareq}
\left\{
  \begin{array}{lll}
L(\phi)=h+\sum\limits_{j=0,1,2,\ldots,3n-1} c_jV^{p-1}Z_j, \ \   {\rm in} \quad D_\ep;\\
\phi=0, \qquad     {\rm on} \quad\partial D_\ep;\\
\int_{D_\ep}V^{p-1}Z_j\phi dy=0, \quad     {\rm for\ all} \ j=0,1,2,\ldots,3n-1.
  \end{array}
  \right.
\end{eqnarray}
In order to perform an invertibility theory for $L$ subject to the above orthogonality conditions, we introduce
some proper weighted $L^\infty$-norms. We start with  for
\begin{equation}\label{starstarn}
\| \psi\|_{**} = \sup_{y\in I_\ve} | |y|^{n-2} \psi (y) | + \sup_{y\in O_\ve} | (1+ |y|^4 ) \psi (y) | ,
\end{equation}
where
\begin{equation}\label{defOI}
I_\ve := \{ y\in D_\ep , |y|<1 \} , \quad O_\ve := \{ y\in D_\ep , |y|>1 \}.
\end{equation}
This $L^\infty$ weighted norm, which allows singularity at $0$, is suitable to estimate the right hand side $h$ in \eqref{lineareq}.
The estimate of $\| E \|_{**}$, where $E$ is the function defined in \eqref{eqv3}, is crucial for our argument, as it will become clear later on. We claim that there exists a positive constant $C$, independent of $\ve $, so that
\begin{equation}\label{eeqv3}
\| E \|_{**} \leq C \ve^{n-2 \over 2}.
\end{equation}
Let us consider first $y \in D_\ve$, with $|y|>1$. Using the result in Lemma \ref{lemma1} and a Taylor expansion, in combination with \eqref{defv0} and \eqref{vanc1}, we immediately see that
$$
\left| E (y) \right| \leq C \ve^{n-2 \over 2} |Q_{\tilde A} (y) |^{p-1} \leq C {\ve^{n-2 \over 2} \over 1+ |y|^4} .
$$
Let us now consider the region $y \in D_\ve$, and $|y|<1$.  In this region, the function $E$ can be estimated as follows
$$
\left| E(y) \right| \leq C \left( {\ve^{n-2 \over 2} \over |y|^{n-2} }\right)^p \leq C {\ve^{n-2 \over 2} \over |y|^{n-2} } ,
$$
since in the region we are considering one has $|y|>\sqrt{\ve}$. With this, \eqref{eeqv3} is proven.
We now introduce an appropriate norm to estimate a solutions to \eqref{lineareq}. This norms depends on the dimension of the space.
For a function $\psi$ defined on $D_\ep$, we define
\begin{eqnarray}\label{starn}
 \| \psi\|_* &=&\sup_{y\in I_\ve }  \left[ | |y|^{\alpha} \psi(y) |
 +||y|^{\alpha+1 }  D \psi(y) |\right] \nonumber \\
 &+&\sup_{y\in O_\ve}  \left[ | ( 1 +|y|^{\beta} )\psi(y) |
 +|( 1 +|y|^{\beta+1} )  D \psi(y) |\right],
\end{eqnarray}
where
\begin{equation}\label{ab}
\alpha = \left\{
  \begin{array}{ll}
n-4  \ & {\rm if} \quad n\geq 5\\
\sigma\  & {\rm if} \quad n = 4
  \end{array}
  \right. , \quad
  \beta = \left\{
  \begin{array}{ll}
\beta = 2 \ & {\rm if} \quad n\geq 5\\
\beta = 2-\sigma\  & {\rm if} \quad n = 4\\
  \end{array}
  \right.
  \end{equation}
  for some $\sigma >0$,
 and
\begin{eqnarray}\label{starn3}
 \| \psi\|_*
 &=&\sup_{y\in D_\ve }  \left[ | ( 1 +|y| )\psi(y) |
 +|( 1 +|y|^{2} )  D \psi(y) |\right],
\end{eqnarray}
if $n=3$.

Equation (\ref{lineareq}) is solved in the following proposition, whose proof is postponed to Section \ref{finitep}.

\begin{proposition}\label{linproppp}
Let $\eta>0$ be fixed as in (\ref{txiaztea}), and assume that the set of parameters $A=(\lambda,\xi, a,\theta)\in \mathbb{R}^+\times \R^n \times \R^2\times \R^{2n-3}$ satisfies (\ref{txiaztea})-(\ref{ztea2n3}). Then there are numbers $\ve_0 >0$,
$C>0$, such that  for all $0<\ve < \ve_0$ and all $h\in C^\alpha
(\bar D_\ve )$, problem (\ref{lineareq}) admits a unique solution
$\phi:=T_\ve (h)$. Moreover,
\begin{eqnarray}\label{phiest}
\|T_\ep (h)\|_* \le C \| h\|_{**}, \quad  |c_{j}|\le C \| h \|_{**},
\end{eqnarray}
and
\begin{eqnarray}\label{phiesta}
\|\nabla_{(d , \tau , a , \theta) } \phi \|_{*}\le C\|h\|_{**}.
\end{eqnarray}
\end{proposition}

\medskip
Based on the results in Proposition \ref{linproppp}, a fixed point argument allows us to solve  the nonlinear problem of finding a function $\phi$ and constants  $c_j$ solutions to
\begin{eqnarray}\label{lineareqnon}
\left\{
  \begin{array}{lll}
L(\phi)=-[N(\phi)+E]+\sum\limits_{j=0,1,\ldots,3n-1} c_jV^{p-1}Z_j, \quad  & {\rm in} \quad D_\ep;\\
\phi=0, \quad  & {\rm on} \quad\partial D_\ep;\\
\int_{D_\ep}V^{p-1}Z_j\phi dy=0, \quad     {\rm for\ all} \ j=0,1,2,\ldots,3n-1.
  \end{array}
  \right.
\end{eqnarray}

The solvability of problem (\ref{lineareqnon}) is established in next Proposition, whose proof is postponed to Section \ref{nonlinear}.

\begin{proposition}\label{linprop}
Assume  the conditions of Proposition \ref{linproppp} are satisfied. Then there are numbers $\ve_0 >0$,
$C>0$, such that  for all $0<\ve < \ve_0$,  there exists a unique solution
$\phi =\phi(d,a,\theta)$ to problem (\ref{lineareqnon}). Moreover, the map $(d,\tau , a,\theta)\to\phi(d, \tau , a,\theta)$ is of class $C^1$ for $\|\cdot\|_\ast$ norm, and
\begin{eqnarray}\label{phiestag}
\|\phi\|_* \le C \ep^{\frac{n-2}{2}},
\end{eqnarray}
and
\begin{eqnarray}\label{phiestah}
\|\nabla_{d , \tau , a , \theta } \phi \|_{*}\le \ep^{\frac{n-2}{2}}.
\end{eqnarray}
\end{proposition}

\medskip

After problem (\ref{lineareq}) has been solved, we find a solution to problem (\ref{eq1tova1}), if we can find a point $(d , \tau , a , \theta)$ such that coefficients $c_{j}$ in (\ref{lineareq}) satisfy
\begin{eqnarray}\label{u5.1}
c_{j}=0\ \ \ \ {\rm for\ all}\ j=0,1,2, \ldots,3n-1.
\end{eqnarray}
For notational convenience, we introduce the set
\begin{equation}
{\mathcal A} := \{ (d , \tau , a , \theta) \, : \, {\mbox {conditions \eqref{txiaztea}--\eqref{ztea2n3} are satisfied}} \} \subset \mathbb{R}^+\times \R^n \times \R^2 \times \R^{2n-3}.
\end{equation}

We now introduce the finite dimensional restriction $F_{\ep}(d ,\tau , a , \theta):{\mathcal A} \to\mathbb{R}$, given by
\begin{equation}\label{funj}
F_{\ep}(d, \tau , a , \theta)=I_\ep\left(V(y)+\phi(y)\right),
\end{equation}
with $V$ defined by (\ref{defvfir}) and
$\phi$ is the unique solution to problem (\ref{lineareqnon}) given by
Proposition \ref{linprop}, and $I_\ep$ is the energy functional associated to problem \eqref{eq1tov}, given by
\begin{eqnarray}\label{energy}
I_\ep(v)&=&\frac{1}{2}\int_{D_\ep}|\nabla v |^2dy-\frac{1}{p+1}
\int_{D_\ep}|v|^{p+1}dy.
\end{eqnarray}

\begin{lemma}\label{lemmaty}
If  $(d,\tau , a,\theta)$ is a critical point of
$F_\ep$, then $v(y)=V(y)+\phi(y)$ is a solution of problem
 \equ{eq1tov}.
\end{lemma}

\begin{proof}
We claim that if $(d, \tau , a , \theta)$ if a critical point for $F_\ep$, then
We first differentiate  $F_\ep$ with respect to
$d$, then we have
\begin{align}\label{just}
DI_\ve (V + \phi ) [ {\partial \over \partial d}  Q_{\tilde A}  + o(1)] = 0 &, \quad DI_\ve (V + \phi ) [ {\partial \over \partial \tau_i}  Q_{\tilde A}  + o(1)] = 0, \quad i=1, \ldots , n \\
DI_\ve (V + \phi ) [ {\partial \over \partial \theta_{12}}  Q_{\tilde A}  + o(1)] = 0
&, \quad
DI_\ve (V + \phi ) [ {\partial \over \partial a_j}  Q_{\tilde A}  + o(1)] = 0 , \quad j=1,2  \nonumber \\
DI_\ve (V + \phi ) [ {\partial \over \partial \theta_{1l}}  Q_{\tilde A}  + o(1)] = 0 &, DI_\ve (V + \phi ) [ {\partial \over \partial \theta_{2l}}  Q_{\tilde A}  + o(1)] = 0, \quad l =3, \ldots , n \nonumber
\end{align}
Let us assume the validity of these equalities. From (\ref{lineareqnon}), we have
\begin{eqnarray}\label{lineareqnonii}
D I_\ve (V + \phi ) [ Z_{i} +o(1) ]&=&\sum_{j} c_j
\int_{D_\ep}V^{p-1}Z_j [ Z_{i} +o(1) ] dy
\end{eqnarray}
where
$$
Z_0 = {\partial \over \partial d}  Q_{\tilde A}  , \quad Z_j = {\partial \over \partial \tau_i}  Q_{\tilde A} , \quad j=1, \ldots ,n, \quad Z_{n+1} = {\partial \over \partial \theta_{12}}  Q_{\tilde A}
$$
$$
Z_{n+2} = {\partial \over \partial a_1}  Q_{\tilde A} , \quad Z_{n+3} = {\partial \over \partial a_2}  Q_{\tilde A}
$$
and, for $l=3, \ldots , n$,
$$
Z_{n+l+1} = {\partial \over \partial \theta_{1l}}  Q_{\tilde A} , \quad Z_{2n+l-1} = {\partial \over \partial \theta_{2l}}  Q_{\tilde A}.
$$
Using \eqref{fin1}, \eqref{fin2}, \eqref{fin3}, \eqref{fin4}, a direct computation gives
\begin{eqnarray*}\label{lineareqlin}
\int_{D_\ve}
V^{p-1}Z_{j}  Z_{i}dy=\left\{
  \begin{array}{ll}
\int_{\R^n} |Q|^{p-1}(y) z_i^2(y)dy + O(\ve^{n\over n-2} )   & \hbox{ if }\ i=j;\\[.2cm]
\int_{\R^n} |Q|^{p-1}(y) z_1(y)z_{n+2}(y)dy + O(\ve^{n\over n-2} )   & \hbox{ if }\ i=1,\ j=n+2;\\[.2cm]
\int_{\R^n} |Q|^{p-1}(y) z_2(y)z_{n+3}(y)dy + O(\ve^{n\over n-2} )   & \hbox{ if }\ i=2,\ j=n+3;\\[.2cm]
 O(\ve^{n\over n-2} )   & \hbox{otherwise},
  \end{array}
  \right.
\end{eqnarray*}
where the functions $z_j$ are the ones defined in \eqref{capitalzeta0}, \eqref{capitalzetaj}, \eqref{capitalzeta2}, \eqref{chico1}, \eqref{chico2}.
Therefore, the condition $\nabla_{(d,\tau,a,\theta)} F_\ep(d,\tau , a,\theta)  = 0$ give the  $3n$ conditions
 $$D I_\ve ( V +  \phi) [ Z_{j} ] = 0, \quad j=0, \ldots , 3n-1 ,$$
that give necessarily that  $c_{j} = 0$ for all $j=0,\ldots,3n-1$. This concludes the proof of the Lemma. We shall now prove \eqref{just}. Since the arguments are similar, we prove the first formula in \eqref{just}. Observe that
$$
{\partial \over \partial d} F_\ep (d,\tau , a,\theta) = DI_\ve (V + \phi ) [ {\partial \over \partial d}  V + {\partial \over \partial d}  \phi ] .
$$
From Lemma \ref{lemma1} and \eqref{defv0},
$$
{\partial \over \partial d} V (y)=
{\partial \over \partial d}  Q_{\tilde A} (y) + \ve^{n-2 \over 2} (1+ {1\over |y|^{n-2}} ) \Theta_{\tilde A} (y),
$$
where $\Theta_{\tilde A} (y)$ is uniformly bounded as $\ve \to 0$. Now, observe that
$$
{\partial \over \partial d}  Q_{\tilde A} (y)= d^{-{\frac{n-4}{2}}}   \left| {y - d \tau  \over |y - d \tau |} - R_\theta a {|y - d \tau| \over d } \right|^{2-n}   z_0 \left( {  {y - d \tau \over d} - R_\theta a |{y - d \tau \over d }|^2 \over 1-2 R_\theta a \cdot  {y - d \tau \over d}  + |a|^2 |{y - d \tau \over d}|^2 }\right).
$$
 Taking into account that $ \|  {\partial \over \partial d}  \phi  \|_* = o(1)$, as $\ve \to 0$, we get that
 $  {\partial \over \partial d } F_\ep = DI_\ve (V + \phi ) [ {\partial \over \partial d}  Q_{\tilde A} + o(1) ] $, as $\ve \to 0$.

\end{proof}

\begin{lemma}\label{largele2}
Assume  the conditions of Proposition \ref{linproppp} are satisfied. Then we have the following expansion
$$
F_\ep(d,\tau , a,\theta) -I_\ep(V)=o(\ep^{\frac{n-2}{2}})\Theta,
$$
where $\Theta$ is $C^1$ uniformly bounded, independent of $\ep$.
\end{lemma}

\begin{proof}
By a Taylor expansion and the fact that $D I_\ep(V+\phi)[\phi]=0$, we have
\begin{eqnarray*}
F_\ep (d,\tau , a,\theta) -I_\ep(V)
&=& I_\ep\left(V+\phi\right)-I_\ep(V)=\int_0^1D^2I\left(V+t\phi\right)[\phi,\phi]t\ dt\nonumber\\
&=&\int_0^1\int_{D_\ep}\left[|\nabla\phi|^2-p(V+t\phi)^{p-1}\phi^2 \right]t\ dt.
\end{eqnarray*}
From (\ref{lineareqnon}), we have
\begin{eqnarray}\label{large2}
F_\ep (d,\tau , a,\theta) -I_\ep(V)
&=&\int_0^1 \int_{D_\ep}\left(p\left[V^{p-1}-(V+t\phi)^{p-1}\right]\phi^2\right.+\left[N(\phi)+E\right]\phi dy\nonumber\\
&\leq& C\int_{D_\ep}|V^{p-1}-(V+\phi)^{p-1}|\phi^2\ dy+\int_{D_\ep}|E|\ |\phi|\ dy +\int_{D_\ep}|N(\phi)|\ |\phi|\ dy\nonumber\\
&=&o(\ep^{\frac{n-2}{2}})\ \Theta,
\end{eqnarray}
uniformly with respect to $(d,\tau , a,\theta)$ in the considered region, where   $\Theta$ is uniformly bounded, independent of $\ep$. Here we used the facts $\|E\|_\ast\leq C\ep^{\frac{n-2}{2}}$ and $\|\phi\|_\ast\leq C\ep^{\frac{n-2}{2}}$.

By a similarly way, using the facts $\|\nabla_{(d,\tau ,a,\theta)} E\|_\ast\leq C\ep^{\frac{n-2}{2}}$ and $\|\partial_{(d,\tau ,a,\theta)}\phi\|_\ast\leq C\ep^{\frac{n-2}{2}}$, we can obtain
$$
\nabla_{(d,\tau ,a,\theta)}\left(F_\ep (d,\tau , a,\theta) -I_\ep(V)\right)=o(\ep^{\frac{n-2}{2}})\Theta.
$$
This ends the proof of Lemma.
\end{proof}

\noindent {\bf Proof of Theorem \ref{main}.}
By Lemma \ref{lemmaty}, we know that $u(\sqrt{\ep }y)=\ep^{-\frac{1}{p-1}}\left(V(y)+\phi(y)\right)$ is a solution to problem (\ref{eq1}) if and only if $(d,\tau , a,\theta)$ is a critical point of $F_\ep(d,\tau ,a,\theta)$. So we have to prove the existence of the critical point of $F_\ep (d,\tau , a,\theta)$.
We observe that, under the change of variables (\ref{defv}), we have
$
I_\ep(v)=J_\ep(u).
$
From Lemma \ref{largele2}, Proposition \ref{energyes},  (\ref{energy0}) and \eqref{energy1} we find
\begin{eqnarray}\label{energy0yu}
F_\ep (d,\tau , a,\theta) = c_1-\Psi (d,\tau , a,\theta) \ep^{\frac{n-2}{2}} +o(\ep^{\frac{n-2}{2}}) \Theta (d,\tau , a,\theta),
\end{eqnarray}
where $\Psi$ is defined as
$$
\Psi (d,\tau , a,\theta) =\frac{1}{2}\left[\gamma_n^{-2} Q(-R_\theta a)^2 H(0,0) d^{n-2} + \frac{c_2}{d^{n-2}}  F(\tau , a , \theta) \right],
 $$
 with $F$ given in \eqref{Fdef} and $\Theta$ is a smooth function of its variables, which is uniformly bounded, together with its first derivatives, as $\ve \to 0$ for $(\lambda,\xi, a,\theta)$ satisfying  (\ref{txiaztea})-(\ref{ztea2n3}).
Thus our result is proven provided we  find a critical point, stable under $C^1$ perturbation,  of the function $\Psi$.

\medskip
\noindent
Firstly, we observe that
\begin{eqnarray*}
\partial_d\Psi  (d,\tau , a,\theta) = \frac{ n-2}{2}\gamma_n^{-2} Q(R_\theta a)^2 H(0,0) d^{n-3} -\frac{ n-2}{2}\frac{c_2}{d^{n-1}} F(\tau , a ,\theta).
\end{eqnarray*}
We have that $\partial_d\Psi_\ep(d_0,\tau , a,\theta)=0$ with $d_0=\left(\frac{c_2  F(\tau , a ,\theta)}{\gamma_n^{-2} Q(R_\theta a)^2 H(0,0)}\right)^{\frac{1}{2n-4}}$, and $\partial_{dd}^2\Psi_\ep(d_0,\tau , a,\theta)>0$.
Moreover, for $d$ any and $\theta$,
$$
(\tau , a) \mapsto \Psi (d, \tau , a , \theta ) \quad {\mbox {has a nondegenerate maximum at}} \quad (0,0).
$$
Lastly,
since $\theta \in\mathcal{O}$
with $\mathcal{O}$  is a compact manifold of dimension $3n$ with no boundary, then, for any $d$, $\tau$, and $a$, the function $\theta \to \Psi (d, \tau , a, \cdot )$ has a stable  minimum $\bar \theta (d, \tau , a)$. Let $\bar \theta_0 = \bar \theta (d_0 , 0 , 0 )$.
We conclude that $(d_0 , 0 , 0 , \bar \theta_0 )$ is a stable critical point for $\Psi$.
Then there exists a critical point $(d_\ep,\tau_\ep, a_\ep,\theta_\ep)$ of $F_\ep$ satisfying $(d_\ep,\tau_\ep,a_\ep,\theta_\ep)\to (d_0,0,0, \bar \theta_0)$ as $\ep\to0$. This concludes the proof of Theorem \ref{main}.

\medskip
\noindent
The rest of the paper is devoted to prove in details all the facts stated until now.

\section{The linear problem: proof of Proposition \ref{linproppp}}\label{finitep}

\begin{proof}[Proof of Proposition \ref{linproppp}:]
The first part of the proof consists in establishing the  priori estimate (\ref{phiest}).
We do it by contradiction:
assume that  there exists a sequence $\ve = \ve_l\to 0$ such that there are functions
$\phi_\ve$ and $h_\ve$  such that
\begin{eqnarray}\label{lineareqlin}
\left\{
  \begin{array}{lll}
L(\phi_\ep)=h_\ep+\sum\limits_{j=0,1,\ldots,3n-1} c_{j}V^{p-1}Z_j, \quad  & {\rm in} \quad D_\ep;\\
\phi_\ep=0, \quad  & {\rm on} \quad\partial D_\ep;\\
\int_{D_\ep}V^{p-1}Z_j\phi_\ep dy=0, \quad     {\rm for\ all} \ j=0,1,2,\ldots,3n-1,
  \end{array}
  \right.
\end{eqnarray}
for certain constants $c_{j}$, depending  on $\ve$, with $\| h_\ve \|_{**} \to 0 $ while $\| \phi_\ve \|_*$ remains bounded away from $0$ as $\ve \to 0$.

We first establish the slightly weaker assertion that
\begin{eqnarray*}
\| \phi_\ve \|_\rho  \to 0
\end{eqnarray*}
with $\rho >0$ a small fixed number, where
\begin{eqnarray*}
 \| \psi\|_\rho &=&\sup_{y\in I_\ve }  \left[ | |y|^{\alpha +\rho} \psi(y) |
 +||y|^{\alpha+1 +\rho }  D \psi(y) |\right] \nonumber \\
 &+&\sup_{y\in O_\ve}  \left[ | ( 1 +|y|^{\beta -\rho } )\psi(y) |
 +|( 1 +|y|^{\beta+1 -\rho} )  D \psi(y) |\right],
\end{eqnarray*}
if $n\geq 4$, and
\begin{eqnarray*}
 \| \psi\|_\rho
 &=&\sup_{y\in D_\ve }  \left[ | ( 1 +|y|^{1-\rho} )\psi(y) |
 +|( 1 +|y|^{2-\rho} )  D \psi(y) |\right],
\end{eqnarray*}
if $n=3$.
To do this, we assume the opposite, so that with no loss of generality we may take
$\|\phi_\ve  \|_\rho  = 1$.

\medskip
We claim that
\begin{equation}\label{ne1}
c_j \to 0, \quad {\mbox {as}} \quad \ve \to 0.
\end{equation}
Testing equation (\ref{lineareqlin}) against
$Z_{i}$, integrating by parts twice we get that
\begin{eqnarray}\label{lino1}
\sum\limits c_{j} \int_{D_\ve}V^{p-1}Z_{j}  Z_{i} dy& = & \int_{D_\ve } [ \Delta  Z_{i} + pV^{p-1} Z_{i}] \phi_\ep dy
-\int_{\partial D_\ve} Z_i \partial_\nu \phi_\ep
  -   \int_{D_\ve } h_\ve Z_{i}dy .
\end{eqnarray}
We claim that
\begin{eqnarray}\label{lineareqlin}
\int_{D_\ve}
V^{p-1}Z_{j}  Z_{i}dy=\left\{
  \begin{array}{ll}
\int_{\R^n} |Q|^{p-1}(y) z_i^2(y)dy + O(\ve^{n\over n-2} )   & \hbox{ if }\ i=j;\\[.2cm]
\int_{\R^n} |Q|^{p-1}(y) z_1(y)z_{n+2}(y)dy + O(\ve^{n\over n-2} )   & \hbox{ if }\ i=1,\ j=n+2;\\[.2cm]
\int_{\R^n} |Q|^{p-1}(y) z_2(y)z_{n+3}(y)dy + O(\ve^{n\over n-2} )   & \hbox{ if }\ i=2,\ j=n+3;\\[.2cm]
 O(\ve^{n\over n-2} )   & \hbox{otherwise},
  \end{array}
  \right.
\end{eqnarray}
\begin{equation}\label{ne2}
\int_{D_\ve }[\Delta  Z_{i} + pV^{p-1}
Z_{i} ] \phi_\ep  = o(1) \|\phi_\ep\|_{\rho} , \quad {\mbox {and}} \quad
\int_{\partial D_\ve} Z_i \partial_\nu \phi_\ep  = o(1) \|\phi_\ep\|_{\rho}
\end{equation}
\begin{equation}
\label{ne3} | \int_{D_\ve } h_\ve  Z_{i } | \leq C \| h_\ve \|_{**} .
\end{equation}
Thus, we conclude that \be |c_{ j} | \leq C \| h_\ve \|_{**}
+o(1)\|\phi_\ve\|_\rho \label{cij}\ee from which \eqref{ne1} readily follows.

\medskip
\noindent
{\it Proof of \eqref{lineareqlin}}. \ \ From \eqref{defv0} and \eqref{neve1}, we observe that
$$
\int_{D_\ep} V^{p-1}Z_{j}  Z_{i}dy =\int_{D_\ve} Q_{\tilde A}^{p-1} \Theta_{\tilde A} [z_j]  \Theta_{\tilde A} [z_i]  \, dy + O(\ve^{n \over n-2} ),
$$
where $\tilde A = (d, d\tau , a , \theta)$. Using first the change of variable $z= {y-d\tau \over d}$, and then the change of variables
$\eta = {z \over |z|^2} $, we have
\begin{eqnarray*}
\int_{D_\ve} Q_{\tilde A}^{p-1} \Theta_{\tilde A} [z_j]  \Theta_{\tilde A} [z_i]  \, dy &=&
\int \left| {z \over |z|} - R_\theta a |z| \right|^{-2n} \left( Q^{p-1} z_j z_i \right) \left( {z- R_\theta a |z|^2 \over |{z \over |z|} - R_\theta a |z| |^2 }\right) \, dz \\
&=& \int |z|^{-2n} \left| {z \over |z|^2} - R_\theta a  \right|^{-2n} \left( Q^{p-1} z_j z_i \right) \left( {{z \over |z|^2} - R_\theta a  \over |{z \over |z|^2} - R_\theta a|^2}\right)    \, dz\\
&=& \int \left( Q^{p-1} z_j z_i \right) ( \eta - R_\theta a) \, d\eta\\
& =& \int_{\R^n} Q^{p-1} z_j z_i  + O(\ve^{n\over n-2} ) .
\end{eqnarray*}
The last equality follows from Lemma \ref{app1} in the Appendix.
This concludes the proof of \eqref{lineareqlin}.

\medskip
\noindent
{\it Proof of \eqref{ne2}}. \ \ We start from the first estimate. Let $g_i = \Delta  Z_{i} + pV^{p-1}
Z_{i}$. By definition of $Z_i$, we have $ g_i= p \left( V^{p-1} - Q_{\tilde A}^{p-1} \right) Z_i$. A close analysis of the functions $z_j$ in \eqref{capitalzeta0}, \eqref{capitalzetaj}, \eqref{capitalzeta2},
\eqref{chico1}, \eqref{chico2} gives that, for some constant $C$,
$$
|z_j (y) | \leq {C \over 1+ |y|^{n-2}} , \quad y \in \R^n.
$$
From Lemma \ref{lemma1}, we see that
\begin{eqnarray*}
V(y) &=& Q_{\tilde A} (y) - \gamma_n^{-1} Q(-R_\theta a) H(\sqrt{\ve} y , \xi ) \ve^{n-2 \over 2} -
F(\tau , a , \theta ) {\ve^{n-2 \over 2} \over |y|^{n-2}} + R(y)
\end{eqnarray*}
with
$$
\left| R(y) \right| \leq c \ve^{n-2 \over 2} \left[ {\ve^{n-2 \over 2} (1+ \ve \la^{-n+1} ) \over |y|^{n-2}} +\la^2 + \ve^{n-2 \over 2} \right].
$$
Thus we have the following estimate
\begin{equation}\label{sil1}
\left| [\Delta  Z_{i} + pV^{p-1}
Z_{i} ] \phi_\ve \right| \leq C Q_{\tilde A}^{p-2} \left( \ve^{n-2 \over 2} + {\ve^{n-2 \over 2} \over |y|^{n-2} } \right) |Z_i \phi_\ve |.
\end{equation}
To estimate $\int_{D_\ve} g_i \phi$, we estimate separately $\int_{I_\ve} g_i \phi$, and $\int_{O_\ve} g_i \phi$.
Consider first the case $p\geq 2$.
In dimensions $5$ and $6$, one has
\begin{eqnarray*}
\left|\int_{I_\ve}   [\Delta  Z_{i} + pV^{p-1}
Z_{i} ] \phi_\ve \right| &\leq& C \ve^{n-2 \over 2} \| \phi \|_\rho \int_{I_\ve} {1\over |y|^{2n-6 +\rho} } \, dy \\
&\leq &C \ve^{n-2 \over 2} \ve^{-{n\over 2} + 3 -{\sigma\over 2} -{\rho \over 2} } \| \phi_\ve \|_\rho \int_{\sqrt{\ve} < |y| <1} {1\over |y|^{n-\sigma} } \, dy \\
&\leq & C \ve^{1-{\sigma \over 2} -{\rho \over 2}} \| \phi \|_\rho .
\end{eqnarray*}
Moreover,
\begin{eqnarray*}
\left|\int_{O_\ve}   [\Delta  Z_{i} + pV^{p-1}
Z_{i} ] \phi_\ve \right| &\leq& C \ve^{n-2 \over 2} \| \phi \|_\rho \int_{O_\ve} {1\over 1+ |y|^{6-\rho}} \, dy = \ve^{2-{\rho \over 2}} \| \phi \|_\rho.
\end{eqnarray*}
In analogous way, one has $$
\left|\int_{I_\ve}   [\Delta  Z_{i} + pV^{p-1}
Z_{i} ] \phi_\ve \right| \leq C \left\{
  \begin{array}{ll}
 \ve^{2-{\rho \over 2} -{\sigma \over 2} } \| \phi \|_\rho & \hbox{ if }\, n=4;\\[.5cm]
  \ve^{1\over 2} \| \phi \|_\rho & \hbox{if} \quad  n=3,
  \end{array}
  \right.
$$
and
$$
\left|\int_{O_\ve}   [\Delta  Z_{i} + pV^{p-1}
Z_{i} ] \phi_\ve \right| \leq C \left\{
  \begin{array}{ll}
 \ve\| \phi \|_\rho & \hbox{ if }\, n=4;\\[.5cm]
  \ve^{1\over 2} \| \phi \|_\rho & \hbox{if} \quad n=3,
  \end{array}
  \right.
$$
as $\ve \to 0$, in dimensions $4$ and $3$.
Similar estimates hold also in dimensions $4$ and $3$. Thus the first estimate in \eqref{ne2} holds true in dimensions $3$ to $6$.

Let us consider now $n\geq 7$, that is $p<2$. Define $R_\ve = \{ y \in D_\ve \, : \, |Q_{\tilde A} (y)| \leq \ve \}$.
We have
$$
\left| [\Delta  Z_{i} + pV^{p-1}
Z_{i} ] \phi_\ve  \right| \leq C \ve^{p-1} Z_j |\phi_\ve |, \quad {\mbox {in}} \quad R_\ve
$$
and
$$
\left| [\Delta  Z_{i} + pV^{p-1}
Z_{i} ] \phi_\ve  \right|\leq C  |Q_{\tilde A}|^{p-1} \ve^{-1} \left( \ve^{n-2 \over 2} + {\ve^{n-2 \over 2} \over |y|^{n-2}} \right)
|Z_j \phi_\ve |, \quad {\mbox {in} } \quad D_\ve \setminus R_\ve
$$
Thus, we get
$$
\left| \int_{I_\ve \cap R_\ve} g_i \phi \right| \leq C\ve^{p-1} \| \phi \|_\rho, \quad
\left| \int_{O_\ve \cap R_\ve} g_i \phi \right| \leq C\ve^{p-1} \| \phi \|_\rho,
$$
and
$$
\left| \int_{I_\ve \cap R_\ve^c} g_i \phi \right| \leq C\ve^{1-{\rho \over 2} - a} \| \phi \|_\rho, \quad
\left| \int_{O_\ve \cap R_\ve^c } g_i \phi \right| \leq C\ve^{{n-4 \over 2}} \| \phi \|_\rho,
$$
for some $a>0$ small. Thus we get the validity of the first estimate in \eqref{ne2}. Let us discuss now the second estimate in \eqref{ne2}.
We write
$$
\int_{\partial D_\ve}  Z_i {\partial \phi \over \partial \nu} \, =
\int_{\partial D_\ve \cap I_\ve }  Z_i {\partial \phi \over \partial \nu} \,
+ \int_{\partial D_\ve \cap O_\ve}  Z_i {\partial \phi \over \partial \nu} .
$$
We observe that
$$
\left| \int_{\partial D_\ve \cap I_\ve }  Z_i {\partial \phi \over \partial \nu} \right| \leq C \| \phi \|_\rho  \int_{\partial B(0, \sqrt{\ve} )}
{1\over |y|^{n-3+\rho} } \leq C \ve^{1-{\rho \over 2}} \| \phi \|_\rho,
$$
and
$$
\left| \int_{\partial D_\ve \cap O_\ve }  Z_i {\partial \phi \over \partial \nu} \right| \leq C \| \phi \|_\rho  \int_{\partial D_\ve \cap O_\ve }
\ve^{n-2 \over 2} \ve^{3-\rho \over 2} \leq C \ve^{1-{\rho \over 2}} \| \phi \|_\rho.
$$
The second estimate in \eqref{ne2} is thus proven.

\medskip
\noindent
{\it Proof of \eqref{ne3}}. We directly see that
$$
\left| \int_{D_\ve} h_\ve Z_i \right| \leq C\left ( \int_{I_\ve} {dy \over |y|^{n-2} } + \int_{O_\ve} {1\over (1+ |y|^{n+2} )} \, dy
\right) \| h \|_{**} \leq C \| h \|_{**}.
$$

\medskip
\noindent
Let  $G_\ve$ denotes the Green's function of $D_\ve$. We have for $x \in D_\ve $
\begin{eqnarray}\label{nonl4}
\phi_\ve (x)&=& \underbrace{p\int_{D_\ve} G_\ve (x,y)  V^{p -1}
\phi_\ve dy }_{:=g_1}   - \underbrace{\int_{D_\ve} G_\ve (x,y) h_\ve \, dy }_{:=g_2}  -  \underbrace{ \sum\limits_{j}  c_{j}\int_{D_\ve}
V^{p-1}Z_{j} G_\ve (x,y)\, dy}_{:=g}
\end{eqnarray}
We claim that
\begin{equation}\label{sil0}
|\phi_\ve (x) |\leq C \left\{
  \begin{array}{ll}
 { \| \phi_\ve \|_\rho + \| h_\ve \|_{**} \over |x|^{n-4}}  & \hbox{ if }\, n\geq 5;\\[.5cm]
 {\| \phi_\ve \|_\rho + \| h_\ve \|_{**} \over |x|^{\sigma}}  & \hbox{ if }\, n=4;\\[.5cm]
\| \phi_\ve \|_\rho + \| h_\ve \|_{**}  & \hbox{if} \quad n=3,
  \end{array}
  \right. \quad {\mbox {for}} \quad |x|<1
\end{equation}
and
\begin{equation}\label{sil00}
|\phi_\ve (x) |\leq C \left\{
  \begin{array}{ll}
 {\| \phi_\ve \|_\rho +   \| h_\ve \|_{**} \over  1 +|x|^2}  & \hbox{ if }\, n\geq 5;\\[.5cm]
 { \| \phi_\ve \|_\rho +  \| h_\ve \|_{**} \over  1 +|x|^{2-\sigma}}  & \hbox{ if }\, n= 4;\\[.5cm]
{\| \phi_\ve \|_\rho +  \| h_\ve \|_{**}  \over 1+ |x|} & \hbox{if} \quad n=3,
  \end{array}
  \right.  \quad {\mbox {for}} \quad |x|>1.
\end{equation}
And similarly,
\begin{equation}\label{sil0d}
|\partial_x \phi_\ve (x) |\leq C \left\{
  \begin{array}{ll}
 { \| \phi_\ve \|_\rho + \| h_\ve \|_{**} \over |x|^{n-3}}  & \hbox{ if }\, n\geq 5;\\[.5cm]
 {\| \phi_\ve \|_\rho + \| h_\ve \|_{**} \over |x|^{\sigma +1}}  & \hbox{ if }\, n=4;\\[.5cm]
\| \phi_\ve \|_\rho + \| h_\ve \|_{**}  & \hbox{if} \quad n=3,
  \end{array}
  \right. \quad {\mbox {for}} \quad |x|<1
\end{equation}
and
\begin{equation}\label{sil00d}
|\partial_x \phi_\ve (x) |\leq C \left\{
  \begin{array}{ll}
 {\| \phi_\ve \|_\rho +   \| h_\ve \|_{**} \over  1 +|x|^3}  & \hbox{ if }\, n\geq 5;\\[.5cm]
 { \| \phi_\ve \|_\rho +  \| h_\ve \|_{**} \over  1 +|x|^{3-\sigma}}  & \hbox{ if }\, n= 4;\\[.5cm]
{\| \phi_\ve \|_\rho +  \| h_\ve \|_{**}  \over 1+ |x|^2} & \hbox{if} \quad n=3,
  \end{array}
  \right.  \quad {\mbox {for}} \quad |x|>1.
\end{equation}
Assume for the moment the validity of these estimates.
Since $\rho$ is arbitrarily small and $\|\phi_\ve  \|_\rho  = 1$, the estimates above imply that that
 $\|\phi_\ve  \|_{L^\infty (B (0,R_1 ) \setminus B(0,R_2) )} >\gamma $ for  certain $R_1 > R_2 >0$ and $\gamma >0$ independent of $\ve$. Then local
elliptic estimates and the bounds above  yield that, up to a
subsequence, $\phi_\ve $
converges uniformly over compacts of $\R^n$ to a nontrivial
solution $\tilde \phi$ of
\begin{eqnarray}
\label{limeq1}
 \Delta \tilde
\phi + p|Q|^{p-1} \tilde \phi  = 0 ,
\end{eqnarray}
which besides satisfies
\begin{eqnarray}
\label{limeq2}
|\tilde \phi (x)| \le C|x|^{-\beta} , \quad \beta = 2, \quad {\mbox {if}} \quad n\geq 5, \quad \beta = 2-\sigma , \quad {\mbox {if}} \quad n=4
\end{eqnarray}
In dimension $n=3$ this means
$
|\tilde \phi (x) | \leq C |x|^{2-n} .
$
In higher dimension, a bootstrap argument of $\tilde \phi $
solution of \equ{limeq1}, using estimate
\equ{limeq2}, gives
$
|\tilde \phi (x) | \leq C |x|^{2-n} .
$
Thanks to non degenerate result in \cite{mw},  this implies that $\tilde \phi$ is a linear
combination of the functions $ z_j$, defined in \equ{capitalzeta0}, \equ{capitalzetaj}, \equ{capitalzeta2}, \equ{chico1} and \equ{chico2}. On the other
hand, dominated convergence Theorem gives that the orthogonality conditions
$ \int_{D_\ve}
 \phi_\ve V^{p-1}  Z_{  j }  = 0
$
pass to the limit, thus getting
$$
 \int_{\R^n} |Q|^{p-1} z_j \tilde \phi
= 0 \quad {\mbox {for all}} \quad j =0, 1,\ldots , 3n-1.
$$
Hence the only
possibility is that $\tilde \phi \equiv 0$, which is a
contradiction which yields the proof of $\|\phi_\ve\|_\rho \to 0$.
Moreover,  we observe that
$$
\|\phi_\ve \|_{*} \le C(\|h_\ve \|_{**} + \|\phi_\ve \|_{\rho} ),
$$
hence $ \|\phi_\ve \|_{*} \to 0 . $

\medskip
\noindent
We shall show the validity of \eqref{sil0} and \eqref{sil00}. To get \eqref{sil0d} and \eqref{sil00d}, one proceeds in a similar way, using
the fact the function $\phi_\ve$ is of class
$C^1$ and
\begin{eqnarray}
\partial_{x_s} \phi_\ve (x)&=& p \int_{D_\ve} \partial_{x_s} G_\ve (x,y)  V^{p-1}
\phi_\ve dy
 - \int_{D_\ve } \partial_{x_s } G_\ve (x,y) h_\ve \, dy\nonumber\\
 && - \sum\limits_{j}c_{j} \int_{D_\ve}V^{p-1}Z_{j} \partial_{x_s} G_\ve (x,y)\, dy, \qquad x\in D_\ve . \nonumber
\end{eqnarray}

Using the definitions of the norm in \eqref{starstarn}, we get that, for $|x|\leq 1$,
\begin{eqnarray}\label{sil4}
\left|g_2 (x) \right| & \leq &C \| h_\ve \|_{**} \left( \int_{I_\ve} {1\over |x-y|^{n-2} } {1\over |y|^{n-2}} \, dy +
\int_{O_\ve} {1\over |x-y|^{n-2} } {1\over 1+ |y|^4} \, dy \right) \nonumber \\
&\leq &C
\left\{
  \begin{array}{ll}
 { \| h_\ve \|_{**} \over |x|^{n-4}}  & \hbox{ if }\, n\geq 5;\\[.5cm]
 { \| h_\ve \|_{**} \over |x|^{\sigma}}  & \hbox{ if }\, n=4;\\[.5cm]
\| h_\ve \|_{**}  & \hbox{if} \quad n=3,
  \end{array}
  \right.
\end{eqnarray}
as consequence of \eqref{ape3}, in Lemma \ref{app2}.
Consider now $|x|>1$. In this region we have
\begin{equation}\label{sil5}
\left|g_2 (x) \right|\leq C
\left\{
  \begin{array}{ll}
 { \| h_\ve \|_{**} \over  1 +|x|^2}  & \hbox{ if }\, n\geq 5;\\[.5cm]
 { \| h_\ve \|_{**} \over  1 +|x|^{2-\sigma}}  & \hbox{ if }\, n= 4;\\[.5cm]
{\| h_\ve \|_{**}  \over 1+ |x|} & \hbox{if} \quad n=3,
  \end{array}
  \right.
\end{equation}
as consequence of \eqref{ape4}, in Lemma \ref{app2}.
Arguing similarly, and using \eqref{cij}, we see that, for $|x| < 1$,
\begin{equation}\label{sil6}
\left| \sum\limits_{j}  c_{j}\int_{D_\ve}
V^{p-1}Z_{j} G_\ve (x,y)\, dy \right| \leq C \left\{
  \begin{array}{ll}
 { \| \phi_\ve \|_\rho + \| h_\ve \|_{**} \over |x|^{n-4}}  & \hbox{ if }\, n\geq 4;\\[.5cm]
 \| \phi_\ve \|_\rho + \| h_\ve \|_{**}  & \hbox{if} \quad n=3,
  \end{array}
  \right.
\end{equation}
and, for $|x|>1$,
\begin{equation}\label{sil7}
\left| \sum\limits_{j}  c_{j}\int_{D_\ve}
V^{p-1}Z_{j} G_\ve (x,y)\, dy \right| \leq C \left\{
  \begin{array}{ll}
 {\| \phi_\ve \|_\rho +   \| h_\ve \|_{**} \over  1 +|x|^2}  & \hbox{ if }\, n\geq 5;\\[.5cm]
 { \| \phi_\ve \|_\rho +  \| h_\ve \|_{**} \over  1 +|x|^{2-\sigma}}  & \hbox{ if }\, n= 4;\\[.5cm]
{\| \phi_\ve \|_\rho +  \| h_\ve \|_{**}  \over 1+ |x|} & \hbox{if} \quad n=3,
  \end{array}
  \right.
\end{equation}
In order to estimate $g_1$, we consider first $n\geq 5$. For $|x|\leq 1$, we use \eqref{ape4} to get
\begin{equation}\label{sil8}
\left|g_1 (x) \right|  \leq C \| \phi_\ve \|_\rho \left( \int_{I_\ve} {1\over |x-y|^{n-2} } {dy\over |y|^{n-4+\rho}}  +
\int_{O_\ve} {1\over |x-y|^{n-2} } {dy\over 1+ |y|^{6+\rho}} \right) C { \| \phi_\ve \|_\rho \over |x|^{n-4}} ,
\end{equation}
and \eqref{ape3} to get, for $|x|>1$
\begin{equation}\label{sil9}
\left|g_1 (x) \right|  \leq  C { \| \phi_\ve \|_\rho \over 1+ |x|^{2}} ,
\end{equation}
In a similar way, we have, for $|x|>1$,
\begin{equation}\label{sil10}
\left|g_1 (x) \right|  \leq  C { \| \phi_\ve \|_\rho \over |x|^{\sigma}}   \quad {\mbox {if}} \quad n=4,
\quad \left|g_1 (x) \right|  \leq  C  \| \phi_\ve \|_\rho    \quad {\mbox {if}} \quad n=3
\end{equation}
and
\begin{equation}\label{sil11}
\left|g_1 (x) \right|  \leq  C { \| \phi_\ve \|_\rho \over 1+ |x|^{2-\sigma}}   \quad {\mbox {if}} \quad n=4,
\quad \left|g_1 (x) \right|  \leq  C  { \| \phi_\ve \|_\rho \over 1+ |x|}    \quad {\mbox {if}} \quad n=3.
\end{equation}
Collecting together estimates \eqref{sil4}--\eqref{sil11}, we obtain the validity of \eqref{sil0} and \eqref{sil00}.

\medskip
\noindent
{\bf Step 2:} The existence of solution to \equ{lineareq}.
To do this, let us consider the space
$$
H= \left\{ \phi \in H_0^1(D_\ve) \ | \
\int_{D_\ve } V^{p-1}  Z_{j }  \phi  = 0, \ \forall \, j=0,1,\ldots,3n-1\ \right\}
$$
endowed with the usual inner product $ [\phi , \psi ] = \int_{D_\ve}
\nabla\phi\nabla\psi . $ Problem \equ{lineareq} expressed in weak form
is equivalent to that of finding a $\phi \in H$ such that $$
 [\phi , \psi ] = \int_{D_\ve}  \bigl(  pV^{p  -1}\phi - h
  \bigl)\, \, \psi\,
 \qquad \forall \psi \quad \in H.
$$
With the aid of Riesz's representation theorem, this equation gets
rewritten in $H$ in the operational form
\begin{eqnarray}\label{T}
\phi  = L_\ve (\phi)
+ \tilde h
\end{eqnarray}
with certain $ \tilde h \in H$
 which depends linearly in $h$ and where $L_\ve $ is a compact
 operator in $H$.
 Fredholm's alternative guarantees unique solvability of this
 problem for any $h$ provided that the homogeneous
 equation
$ \phi = T_\ve (\phi) $ has only the zero solution in $H$.  Assume it has a nontrivial
solution $\phi =\phi_\ve$, which with no loss of generality may be
taken so that $\|\phi_\ve \|_* =1$. But for what we proved before, necessarily $\|\phi_\ve \|_*\to 0$. This
is certainly a contradiction that proves that this equation only
has the trivial solution in $H$. We conclude then that for each
$h$, problem \equ{lineareq} admits a unique solution.
Standard arguments give then the validity of \equ{phiest}.

\bigskip

We next analysis the dependence of the solution $\phi$ to \equ{lineareq} on the parameters
$A'  = (d,  \tau,a, \theta)$.
Let us define $A'= (A_{1} , A_{2} , \ldots , A_{3n} )$ the components of the vector
$A'$. Let us differential $\phi$ with respect to $A_{s}$, for some $s=1, \ldots , 3n$.  We set formally
$Z= {\partial \over \partial A_{s}} \phi$.
Then $Z$ satisfies the following equation
$$
\Delta Z+pV^{p-1}Z=-p\partial_{A_{s}}(V^{p-1})\phi+\sum\limits_{j}e_jV^{p-1}Z_j+c_j\partial_{A_{s}}(V^{p-1}Z_j)\quad \mbox{in}\ \ D_\ep
$$
Here $e_j=\partial_{A_{s}}c_j$. Besides, from differentiating the orthogonality condition $\int_{D_\ep}V^{p-1}Z_j\phi dy=0$, we get
$$
\int_{D_\ep}\partial_{A_{s}}(V^{p-1}Z_j)\phi dy +\int_{D_\ep}V^{p-1}Z_jZ dy=0.
$$
Let us consider constants $b_i$ such that
\begin{eqnarray}\label{etbi}
\int_{D_\ep}V^{p-1}Z_jZ-\sum\limits_{i}b_i\int_{D_\ep}V^{p-1}Z_jZ_i=0.
\end{eqnarray}
These relations amount to
$$
\sum\limits_{i}b_i\int_{D_\ep}V^{p-1}Z_jZ_i=\int_{D_\ep}V^{p-1}Z_jZ.
$$
Since this system is diagonal dominant with uniformly bounded coefficients, we use it is uniquely solvable and that
$$
b_i=O(\|\phi\|_\ast).
$$
Let us set $\eta=Z-\sum\limits_{i}b_iZ_j$, thus   $\eta \in H_0^1(D_\ve)$ and
\begin{eqnarray} \label{ortog}
\int_{D_\ve } V^{p-1} Z_{j } \eta = 0\quad\hbox{ for all }j.
\end{eqnarray}
On the other hand, we have that
\begin{eqnarray}\label{et}
\Delta \eta    + p V^{p-1 } \eta   =
 f + \sum_{j}e_{j}V^{p-1}  Z_{ j} \quad\hbox{in }
D_\ve ,
\end{eqnarray}
where $e_{j} = {\partial \over \partial A_{s}} c_{ j}$ and
\begin{eqnarray}\label{f}
f= \sum_{j} b_{j} (
-(\Delta + pV^{p-1}) Z_{j} + c_{j}\partial_{A_{s}}
(V^{p-1} Z_{ j }) -
 p \partial_{A_{s}} (V^{p-1 }  \phi ) ,
\end{eqnarray}
Thus  we have that
$ \eta = T_\ve (f) $. Moreover,
we easily see that
$$\|\phi\partial_{A_{s}} (V^{p-1 })\|_{**}
\le C \|\phi \|_{*} .
$$
On the other hand
$$
|\partial_{A_{s} } (V^{p-1 }Z_{ i  }(x))| \le C|x |^{-n-4} ,
$$
hence
$$
\| c_{i}\partial_{A_{s}}V^{p-1} Z_{ i } \|_{**} \le C \|h\|_{**}
$$
since we have that $c_{ i } = O(\| h \|_{**} ) $.
We conclude that
$$
\|f\|_{**} \le C \|h\|_{**} .
$$
Reciprocally, if we  define
$$
Z= T_\ve (f)  + \sum_{j}e_{ j} V^{p-1} Z_{j},
$$
with $b_{j}$ given
by relations \equ{etbi} and $f$ by \equ{f}, we check that indeed $Z=\partial_{A_{s}}\phi$. In fact $Z$
depends continuously on the parameters $A'$ and  $h$
for the norm $\|\ \|_{*}$, and $\|Z\|_{*} \le C\|h\|_{**}$ for parameters
in the considered region.

In other words, we proved that  $(d,\tau, a,\theta) \mapsto T_\ve$ is of class
$C^1$ in ${\mathcal L} ( L^{\infty}_{**} , L^{\infty}_{*})$ and, for
instance,
\begin{equation*}
(D_{A_{s}} T_\ve) (h) = T_\ve (f)  + \sum  b_{ j}
Z_{j}, \label{dl}
\end{equation*}
where $f$ is given by \equ{f} and $b_{j}$
by \equ{etbi}.
This concludes the proof.
\end{proof}

\section{The non-linear Problem: proof of Proposition \ref{linprop}} \label{nonlinear}

\begin{proof}[Proof of Proposition \ref{linprop}:]
We write the equation in \equ{lineareqnon}
as
$$
 \Delta \phi  + pV^{p  -1} \phi = -N  ( \phi )-E  + \sum_{j}
c_{ j} V^{p-1}  Z_{  j } \quad \mbox{ in } \quad D_\ve
$$
where $N(\phi)$ and $E$ are defined respectively by (\ref{eqv1}) and (\ref{eqv3}).
We already showed in \eqref{eeqv3} that
$
\| E \|_{**} \leq C \ep^{\frac{n-2}{2}} .
$
To estimate $N (\phi)$, it is convenient, and
sufficient for our purposes, to assume $\|\phi  \|_* <1$. Note
that, if $n\leq 6$, then
$p\ge 2$ and we can estimate
$$
|N (\phi ) | \leq C |V|^{p-2} |\phi|^2
$$
and hence
$$
| N  (\phi ) | (x) \leq C
\left\{
  \begin{array}{ll}
 { \| \phi      \|_{*}^2  \over |x|^{2}}  & \hbox{ if }\, n\geq 4;\\[.5cm]
\| \phi      \|_{*}^2  & \hbox{if} \quad n=3,
  \end{array}
  \right. \quad {\mbox {for}} \quad |x| <1,
$$
and
$$
| N  (\phi ) | (x) \leq C
\left\{
  \begin{array}{ll}
 { \| \phi      \|_{*}^2  \over 1+ |x|^{4}}  & \hbox{ if }\, n\geq 5;\\[.5cm]
   \ve^{-{\sigma \over 2}}{ \| \phi      \|_{*}^2  \over 1+ |x|^{4-\sigma }}  & \hbox{ if }\, n=4;\\[.5cm]
{\| \phi      \|_{*}^2 \over 1+ |x|^4} & \hbox{if} \quad n=3,
  \end{array}
  \right. \quad {\mbox {for}} \quad |x| >1,
$$
Assume now that $n>6$. If $|\phi |\geq {1\over 2} V$ we have
$$
|N(\phi ) | \leq C  |\phi|^p,
$$
and thus
$$
|N(\phi ) (x) | \leq C \left\{
  \begin{array}{ll}
 \ve^{p} { \| \phi \|_*  \over |x|^{n-2} }& \hbox{ if }\, |x|<1;\\[.5cm]
 \ve^{p} { \| \phi \|_*  \over 1+ |x|^{4} }& \hbox{ if }\, |x|>1
  \end{array}
  \right.
$$
Let us consider now the case $|\phi | \leq {1\over 2} V$. In this case, we have that
$|N(\phi ) | \leq C |V|^{p-1} |\phi|$, for some constant $C$.
Thus, for $|x|<1$, we get
$$
|N(\phi )| \leq C {\ve^2 \over |y|^4} {\| \phi \|_* \over |y|^{n-4}} \leq C \ve {\|\phi \|_* \over |y|^{n-2}}
$$
while for $|x|>1$,
$$
|N(\phi ) | \leq C \ve^2 {\| \phi \|_* \over 1+ |x|^2} \leq C \ve {\| \phi \|_* \over 1+ |x|^4} .
$$
Combining these relations we get
\begin{eqnarray}
\| N (\phi )\|_{**} \leq
\left\{ \begin{array}{ll}
                        C \| \phi     \|_{*}^2  &\mbox{ if } n \le 6, \, n \not= 4\\
C \ve^{-{\sigma \over 2}} \| \phi     \|_{*}^2  &\mbox{ if } \, n = 4\\
          C \ve^{-{n-4 \over n-2}} \| \phi  \|_{*}^2  &\mbox{ if } n> 6.
                          \end{array}
                \right.
\label{Ne}
\end{eqnarray}
Now, we are in position to prove that problem (\ref{lineareqnon}) has a
unique solution $\phi=\widetilde\phi+\widetilde\psi$, with
\begin{equation}\label{tor}
\widetilde\psi:=  -T_\ve  (E ),
\end{equation} with the required properties.
Here $T_\ve$ denotes the linear operator defined by Proposition \ref{linproppp}, namely $\phi = T_\ve (h)$ solves \eqref{lineareq}.
We see that problem \equ{lineareqnon} is equivalent to solving a fixed point
problem. Indeed $\phi = \tilde \phi +\tilde \psi$ is a solution of
\equ{lineareqnon} if and only if
$$
\tilde \phi = - T_\ve (N (\tilde \phi + \tilde \psi )) \equiv A_\ve (\tilde
\phi ).
$$
We proceed to prove that the operator $A_\ve $ defined above is a
contraction inside a properly chosen region. Since $\| E \|_* \leq C \ve^{\frac{n-2}{2}}$, the result of Proposition \ref{linprop} gives that
$$
\| \tilde \psi \|_{**} \leq C \ve^{\frac{n-2}{2}}
$$
and
\begin{eqnarray}
\|N (\tilde \psi + \eta)\|_{**}
\leq
\left\{ \begin{array}{ll}
                        C \| \eta     \|_{*}^2  &\mbox{ if } n \le 6, \, n \not= 4\\
C \ve^{-{\sigma \over 2}} \| \eta     \|_{*}^2  &\mbox{ if } \, n = 4\\
          C \ve^{-{n-4 \over n-2}} \| \eta  \|_{*}^2  &\mbox{ if } n> 6.
                          \end{array}
                \right.
                \label{Ne1}\end{eqnarray}
Call
$$F = {{\{ \eta \in H_0^1 \, : \, ||\eta ||_{*} \leq
R \ve^{\frac{n-2}{2}} \}}}.
  $$
From Proposition \ref{linprop} and \equ{Ne1} we conclude that, for $\ve$
sufficiently small and any $\eta \in  F$ we have $$ \| A_\ve
(\eta ) \|_{*}  \le  C \ve^{\frac{n-2}{2}} . $$ If we choose $R$ big enough in the definition of $F$, we get then that $A_\ve$ maps $F$ in itself.
Now we will show that the map $A_\ve $ is a contraction, for any $\ve$
small enough. That will imply that $A_\ve $ has a unique fixed point in
$F$ and hence problem \equ{lineareqnon} has a unique solution.
For any $\eta_1 $, $ \eta_2 $ in $F$ we have
$$ \| A_\ve (\eta_1 ) - A_\ve (\eta_2 ) \|_{*} \leq C  \| N_\ve
(\tilde \psi  + \eta_1 ) - N_\ve (\tilde \psi  + \eta_2 ) \|_{**} , $$
hence we just need to check that $N$ is a contraction in its
corresponding norms. By definition of $N$
$$
D_{\bar\eta} N_\ve  (\bar\eta) = p[ (V+ \bar \eta )_{+}^{p-1}
- V^{p  -1} ].
$$
Arguing as before, we get
 $c\in (0,1)$ such that
$$
\|N  (\tilde \psi  +\eta_1 ) - N ( \tilde \psi  + \eta_2 ) \|_{**}
\le c \|\eta_1 - \eta_2\|_{*}.
$$
This concludes the proof of existence of $\phi$ solution to \equ{lineareqnon}, and the first estimate in \equ{phiestag}.

The regularity of the map $(d,\tau , a,\theta)\mapsto\phi$ can be proved by standard arguments involving the implicit function, and then we get the estimate (\ref{phiestah}), which can be seen in \cite{mw}.
\end{proof}

\section{Appendix} \label{appe}

Let us recall that the existence of $Q$ was obtained in \cite{dmpp1,dmpp2}, and also we will use some facts and computations in \cite{mw}. We have
\begin{align*}
Q(y) =U(y) - \sum_{j=1}^k U_j (y) +\tilde  \phi (y)
\end{align*}
with
\begin{equation*}
U(y) = \alpha_n \left( {1 \over 1+ |y|^2} \right)^{n-2 \over 2}, \quad U_j (y) = \mu_k^{-{n-2 \over 2}} U(\mu_k^{-1} (y-\xi_j )),
\end{equation*}
and  the function $\tilde  \phi$ is described in (\ref{decphi}).
Let us now define the following functions
\begin{align}\label{pioa}
\pi_\alpha(y)=\frac{\partial}{\partial y_\alpha}\tilde  \phi (y),\quad \mbox{for}\ \alpha=1,2,\cdots,n;\quad
\pi_0(y)=\frac{n-2}{2}\tilde  \phi (y)+\nabla \tilde  \phi (y)\cdot y.
\end{align}
Observe
that the function $\pi_0$ is even in each of its variables, namely
$$
\pi_0(y_1,\cdots,y_j,\cdots,y_n)=\pi_0(y_1,\cdots,-y_j,\cdots,y_n),\quad \mbox{for\ all}\ j=1,2,\cdots,n,
$$
while $\pi_\alpha$, for $\alpha=1,2,\cdots,n$ is odd in the $y_\alpha$ variable, while it is even in all the other variables. Furthermore, all functions $\pi_\alpha$ are invariant under rotation of $\frac{2\pi}{k}$ in the first
two coordinates, namely they satisfy (\ref{sim00}).

Now let us define the functions
$$
\mathcal{Z}_0(y)=\frac{n-2}{2}U(y)+\nabla U(y)\cdot y,
$$
and
$$
\mathcal{Z}_\alpha(y)=\frac{\partial}{\partial y_\alpha} U(y),\quad \mbox{for}\ \alpha=1,2,\cdots,n.
$$
We note that, by symmetry and (8.5) in \cite{mw}, we have
\begin{align}\label{defctilde}
\int_{\mathbb{R}^n}U(y)^{p-1}\mathcal{Z}_{0}(y)^2dy=\int_{\mathbb{R}^n}U(y)^{p-1}\mathcal{Z}_{\alpha}(y)^2dy=2^{\frac{n-4}{2}}
n(n-2)^2\frac{\Gamma(\frac{n}{2})^2}{\Gamma(n+2)}
:=\tilde{c},
\end{align}
for $\alpha=1,2,\cdots,n$.

We have the following results.
\begin{lemma}\label{app1}
Let the functions $z_i$ be  defined in \eqref{capitalzeta0}-\eqref{chico2}, and $\mu=\mu_k$ be defined in \eqref{parameters} and satisfies (\ref{parametersas}).
It holds that,
\begin{align}\label{estizon1tong}
  \int_{\R^n} |Q (y)|^{p-1} z_i (y)^2 \, dy
=
(k+1)\tilde{c} +
\left\{
\begin{array}{ll}
O(k^{ (1-\frac{n}{q})\max\{1,\frac{4}{n-2}\} } ), \quad  &\mbox{ if }\quad n\geq 4,\\[1mm]
O( |\log k|^{-1}), \quad  &\mbox{ if }\quad n=3,
\end{array}
\right.
\end{align}
for $i=0,1,2,\cdots,3n-1$, where $\frac{n}{2}<q<n$ and $\tilde{c}$ is as in (\ref{defctilde}). Moreover, there exists $C>0$ such that
\begin{equation}\label{ape2}
\left| z_i (y) \right| \leq C {1\over 1+ |y|^{n-2}} , \quad {\mbox {for}} \quad y \in \R^n.
\end{equation}
\end{lemma}

\begin{proof}
We will give the proof for the case $i=0$ in (\ref{estizon1tong}), and the others can be obtained in the same way. Moreover, (\ref{ape2}) follows directly from the definition of $z_i$ and the results of Proposition 2.1 in \cite{mw}.
We have
\begin{align}\label{estizon1tongoo}
& \int_{\R^n} |Q (y)|^{p-1} z_0 (y)^2 \, dy \nonumber\\
=&\int_{\R^n} |Q (y)|^{p-1}\Big[\frac{n-2}{2}Q(y)+\nabla Q(y)\cdot y\Big]^2 \, dy\nonumber\\
=&\int_{\R^n} |Q (y)|^{p-1}\Big[\Big(\frac{n-2}{2}U(y)+\nabla U(y)\cdot y\Big)\nonumber\\
&\qquad \qquad \qquad-\sum\limits_{j=1}^k\Big(\frac{n-2}{2}U_j(y)+\nabla U_j(y)\cdot y\Big)+\pi_0(y)\Big]^2 \, dy\nonumber\\
=&\int_{\R^n} |Q (y)|^{p-1} \Big(\frac{n-2}{2}U(y)+\nabla U(y)\cdot y\Big)^2dy\nonumber\\
&+\sum\limits_{j=1}^k\int_{\R^n} |Q (y)|^{p-1} \Big(\frac{n-2}{2}U_j(y)+\nabla U_j(y)\cdot y\Big)^2dy \nonumber\\
&+\int_{\R^n} |Q (y)|^{p-1}|\pi_0(y)|^2dy\nonumber\\
&+2\int_{\R^n} |Q (y)|^{p-1}\Big(\frac{n-2}{2}U(y)+\nabla U(y)\cdot y\Big)\nonumber\\
&\qquad \quad \times \Big[-\sum\limits_{j=1}^k\Big(\frac{n-2}{2}U_j(y)+\nabla U_j(y)\cdot y\Big)+\pi_0(y)\Big]dy\nonumber\\
&  +2 \sum\limits_{i\neq j}\int_{\R^n} |Q (y)|^{p-1}\Big(\frac{n-2}{2}U_i(y)+\nabla U_i(y)\cdot y\Big)\Big(\frac{n-2}{2}U_j(y)+\nabla U_j(y)\cdot y\Big)dy\nonumber\\
:=&A_1+A_2+A_3+A_4+A_5.
\end{align}
Next we estimate each term as follows.

{\it Estimate of $A_1$:} We have
\begin{align*}
A_1=&\int_{\R^n} \Big|U(y) - \sum_{j=1}^k U_j (y)+\tilde  \phi (y)\Big|^{p-1} \Big(\frac{n-2}{2}U(y)+\nabla U(y)\cdot y\Big)^2dy\nonumber\\
 = & \int_{\R^n} \Big[|U(y)|^{p-1} + \sum_{j=1}^k |U_j (y)|^{p-1}+|\tilde  \phi (y)|^{p-1}+|U(y)|^{\gamma}|\sum\limits_{j=1}^kU_j(y)+\tilde{\phi}(y)|^{p-1-\gamma}\Big]\mathcal{Z}_0(y)^2dy\nonumber\\
= &\int_{\R^n} |U(y)|^{p-1}\mathcal{Z}_0(y)^2dy+ \sum_{j=1}^k\int_{\R^n} |U_j (y)|^{p-1}\mathcal{Z}_0(y)^2dy\nonumber\\
&+  \int_{\R^n} \Big[ |\tilde  \phi (y)|^{p-1}+|U(y)|^{\gamma}|\sum\limits_{j=1}^kU_j(y)+\tilde{\phi}(y)|^{p-1-\gamma}\Big]\mathcal{Z}_0(y)^2dy.
\end{align*}
Since
\begin{align*}
\sum_{j=1}^k\int_{\R^n} |U_j (y)|^{p-1}\mathcal{Z}_0(y)^2dy=&(n-2)^2\alpha_n^{p+1}\sum_{j=1}^k\int_{\R^n} \frac{\mu_k^2}{(\mu_k^2+|y-\xi_j|^2)^2}\frac{(1-|y|^2)^2}{(1+|y|^2)^n}dy\\
=&(n-2)^2\alpha_n^{p+1}\sum_{j=1}^k\mu_k^{n-2}\int_{\R^n} \frac{1}{(1+|z|^2)^2}\frac{(1-|\mu_kz+\xi_j|^2)^2}{(1+|\mu_kz+\xi_j|^2)^n}dz\\
=&(n-2)^2\alpha_n^{p+1}\sum_{j=1}^k\mu_k^{n-2}\int_{\{|z|\leq\frac{1}{2\mu_k}\}} \frac{1}{(1+|z|^2)^2}\frac{(1-|\mu_kz+\xi_j|^2)^2}{(1+|\mu_kz+\xi_j|^2)^n}dz\\
&+(n-2)^2\alpha_n^{p+1}\sum_{j=1}^k\mu_k^{n-2}\int_{\{|z|\geq\frac{1}{2\mu_k}\}} \frac{1}{(1+|z|^2)^2}\frac{(1-|\mu_kz+\xi_j|^2)^2}{(1+|\mu_kz+\xi_j|^2)^n}dz\\
=&O\Big(\sum_{j=1}^k\mu_k^{n-2}\int_{\{|z|\leq\frac{1}{2\mu_k}\}} \frac{1}{(1+|z|^2)^2}\frac{1}{(1+|\mu_kz+\xi_j|^2)^{n-2}}dz\Big)\\
&+O\Big(\sum_{j=1}^k\mu_k^{n-2}\int_{\{|z|\geq\frac{1}{2\mu_k}\}} \frac{1}{(1+|z|^2)^2}\frac{1}{(1+|\mu_kz+\xi_j|^2)^{n-2}}dz\Big)\\
=&O\Big(\sum_{j=1}^k\mu_k^{n-2}\int_{\{|z|\leq\frac{1}{2\mu_k}\}} \frac{1}{(1+|z|^2)^2} \Big)\\
&+O\Big(\sum_{j=1}^k\mu_k^{n-2}\int_{\{|z|\geq\frac{1}{2\mu_k}\}} \frac{1}{(1+|z|^2)^2}\frac{1}{(1+|\mu_kz |^2)^{n-2}}dz\Big)\\
=&O(k\mu_k^2),
\end{align*}
and by (\ref{decphi})-(\ref{estphi1}), we have
\begin{align*}
|\tilde{\phi}(y)|\leq \frac{C}{(1+|y|)^{n-2}}\left\{
\begin{array}{ll}
k^{ 1-\frac{n}{q} }, \  &\mbox{ if }\  n\geq 4,\\[1mm]
 |\log k|^{- 1 }, \   &\mbox{ if }\  n=3.
\end{array}
\right.
\end{align*}
Then
\begin{align*}
&\Big|\int_{\R^n} \Big[ |\tilde  \phi (y)|^{p-1}+|U(y)|^{\gamma}|\sum\limits_{j=1}^kU_j(y)+\tilde{\phi}(y)|^{p-1-\gamma}\Big]\mathcal{Z}_0(y)^2dy\Big|\\
\leq &C\left\{
\begin{array}{ll}
k^{ (1-\frac{n}{q})\frac{4}{n-2} }, \   &\mbox{ if }\  n\geq 4,\\[1mm]
 |\log k|^{- 4 }, \  &\mbox{ if }\  n=3.
\end{array}
\right.
\end{align*}
Thus
\begin{align}\label{esta1jk}
A_1=& \tilde{c} +
 \left\{
\begin{array}{ll}
O(k^{ (1-\frac{n}{q})\frac{4}{n-2} }), \   &\mbox{ if }\  n\geq 4,\\[1mm]
O(|\log k|^{- 4 }), \   &\mbox{ if }\  n=3.
\end{array}
\right.
\end{align}
Here $\tilde{c}$ is as in (\ref{defctilde}).

\smallskip

{\it Estimate of $A_2$:} We have
\begin{align*}
A_2=&\int_{\R^n} \Big|U(y) - \sum_{j=1}^k U_j (y)+\tilde  \phi (y)\Big|^{p-1} \Big(\frac{n-2}{2}U_j(y)+\nabla U_j(y)\cdot y\Big)^2dy\nonumber\\
 = & \int_{\R^n} \Big[|U(y)|^{p-1} + \sum_{j=1}^k |U_j (y)|^{p-1}+|\tilde  \phi (y)|^{p-1}+|U(y)|^{\gamma}|\sum\limits_{j=1}^kU_j(y)+\tilde{\phi}(y)|^{p-1-\gamma}\Big]\nonumber\\
 &\qquad \times \Big(\frac{n-2}{2}U_j(y)+\nabla U_j(y)\cdot y\Big)^2dy\nonumber\\
= &\int_{\R^n} |U(y)|^{p-1}\Big(\frac{n-2}{2}U_j(y)+\nabla U_j(y)\cdot y\Big)^2dy\nonumber\\
&  + \sum_{j=1}^k\int_{\R^n} |U_j (y)|^{p-1}\Big(\frac{n-2}{2}U_j(y)+\nabla U_j(y)\cdot y\Big)^2dy\nonumber\\
&+  \int_{\R^n} \Big[ |\tilde  \phi (y)|^{p-1}+|U(y)|^{\gamma}|\sum\limits_{j=1}^kU_j(y)+\tilde{\phi}(y)|^{p-1-\gamma}\Big]\Big(\frac{n-2}{2}U_j(y)+\nabla U_j(y)\cdot y\Big)^2dy.
\end{align*}
Since
\begin{align*}
 & \int_{\R^n} |U(y)|^{p-1}\Big(\frac{n-2}{2}U_j(y)+\nabla U_j(y)\cdot y\Big)^2dy\nonumber\\
=&(n-2)^2\alpha_n^{p+1} \sum\limits_{j=1}^k\mu_k^{n-2}\int_{\R^n} \frac{1}{(1+|y|^2)^2}\frac{\mu_k^2+|y-\xi_j|^2-2\sum_{i=1}^n y_i(y-\xi_j)_i}{(\mu_k^2+|y-\xi_j|^2)^n} dy\nonumber\\
=&(n-2)^2\alpha_n^{p+1} \sum\limits_{j=1}^k \int_{\R^n} \frac{1}{(1+|\mu_kz+\xi_j|^2)^2}\frac{1+|z|^2-2\sum_{i=1}^n z_i(z-\frac{\xi_j}{\mu_k})_i}{(1+|z|^2)^n} dy\nonumber\\
=&(n-2)^2\alpha_n^{p+1} \sum\limits_{j=1}^k \int_{\{|z|\leq\frac{1}{2\mu_k}\}} \frac{1}{(1+|\mu_kz+\xi_j|^2)^2}\frac{1+|z|^2-2\sum_{i=1}^n z_i(z-\frac{\xi_j}{\mu_k})_i}{(1+|z|^2)^n} dy\nonumber\\
&+(n-2)^2\alpha_n^{p+1} \sum\limits_{j=1}^k \int_{\{|z|\geq\frac{1}{2\mu_k}\}} \frac{1}{(1+|\mu_kz+\xi_j|^2)^2}\frac{1+|z|^2-2\sum_{i=1}^n z_i(z-\frac{\xi_j}{\mu_k})_i}{(1+|z|^2)^n} dy\nonumber\\
=&O(\mu_k^{n-2})+O\Big(\sum\limits_{j=1}^k \mu_k^{-4}\int_{\{|z|\geq\frac{1}{2\mu_k}\}} \frac{1}{  |z |^4}\frac{1-|z|^2 }{(1+|z|^2)^n} dy\Big)\nonumber\\
=&O(\mu_k^{n-2}).
\end{align*}
and
\begin{align*}
 &\sum\limits_{j=1}^k \int_{\R^n} |U_j(y)|^{p-1}\Big(\frac{n-2}{2}U_j(y)+\nabla U_j(y)\cdot y\Big)^2dy\nonumber\\
=&(n-2)^2\alpha_n^{p+1} \sum\limits_{j=1}^k\mu_k^{2}\int_{\R^n}  \frac{\mu_k^2+|y-\xi_j|^2-2\sum_{i=1}^n y_i(y-\xi_j)_i}{(\mu_k^2+|y-\xi_j|^2)^{n+2}} dy\nonumber\\
=&(n-2)^2\alpha_n^{p+1} \sum\limits_{j=1}^k \int_{\R^n}  \frac{1+|z|^2-2\sum_{i=1}^n z_i(z-\frac{\xi_j}{\mu_k})_i}{(1+|z|^2)^{n+2}} dy\nonumber\\
=&(n-2)^2\alpha_n^{p+1} \sum\limits_{j=1}^k \int_{\{|z|\leq\frac{1}{2\mu_k}\}}  \frac{1+|z|^2-2\sum_{i=1}^n z_i(z-\frac{\xi_j}{\mu_k})_i}{(1+|z|^2)^{n+2}} dy\nonumber\\
&+(n-2)^2\alpha_n^{p+1} \sum\limits_{j=1}^k \int_{\{|z|\geq\frac{1}{2\mu_k}\}}  \frac{1+|z|^2-2\sum_{i=1}^n z_i(z-\frac{\xi_j}{\mu_k})_i}{(1+|z|^2)^{n+2}} dy\nonumber\\
=&(n-2)^2\alpha_n^{p+1} \sum\limits_{j=1}^k \int_{\R^n}  \frac{1-|z|^2 }{(1+|z|^2)^{n+2}} dy+O(\mu_k^{n+2})\nonumber\\
=&k\tilde{c}+O(\mu_k^{n+2}).
\end{align*}
Moreover,
\begin{align*}
&\Big|\int_{\R^n} \Big[ |\tilde  \phi (y)|^{p-1}+|U(y)|^{\gamma}|\sum\limits_{j=1}^kU_j(y)+\tilde{\phi}(y)|^{p-1-\gamma}\Big]\Big(\frac{n-2}{2}U_j(y)+\nabla U_j(y)\cdot y\Big)^2dy\Big|\\
\leq &C\left\{
\begin{array}{ll}
k^{ (1-\frac{n}{q})\frac{4}{n-2} }, \   &\mbox{ if }\  n\geq 4,\\[1mm]
 |\log k|^{- 4 }, \  &\mbox{ if }\  n=3,
\end{array}
\right..
\end{align*}
Therefore, we obtain
\begin{align}\label{esta1jka2}
A_2=& k\tilde{c} +
\left\{
\begin{array}{ll}
O(k^{ (1-\frac{n}{q})\frac{4}{n-2} }), \   &\mbox{ if }\  n\geq 4,\\[1mm]
O(|\log k|^{- 4 }), \   &\mbox{ if }\  n=3.
\end{array}
\right.
\end{align}

\smallskip

{\it Estimate of $A_3$:} We have
\begin{align}\label{esta3}
A_3=&\int_{\R^n} |Q (y)|^{p-1}|\pi_0(y)|^2dy\nonumber\\
\leq& C\int_{\R^n} |Q (y)|^{p-1} \frac{1}{(1+|y|)^{2(n-2)}}dy  \left\{
\begin{array}{ll}
O(k^{2 (1-\frac{n}{q})  }), \   &\mbox{ if }\  n\geq 4,\\[1mm]
O(|\log k|^{- 2 }), \   &\mbox{ if }\  n=3,
\end{array}
\right.\nonumber\\
\leq& C \left\{
\begin{array}{ll}
k^{2 (1-\frac{n}{q})  }, \   &\mbox{ if }\  n\geq 4,\\[1mm]
 |\log k|^{- 2 }, \   &\mbox{ if }\  n=3.
\end{array}
\right.
\end{align}

\smallskip

{\it Estimate of $A_4$:} We have
\begin{align}\label{esta4}
|A_4|
\leq& C \left\{
\begin{array}{ll}
k^{ 1-\frac{n}{q} }, \   &\mbox{ if }\ n\geq 4,\\[1mm]
 |\log k|^{- 1 }, \   &\mbox{ if }\  n=3.
\end{array}
\right.
\end{align}

\smallskip

{\it Estimate of $A_5$:} We have
\begin{align}\label{esta5}
|A_5|
\leq& C   k\mu_k^2.
\end{align}
Thus (\ref{estizon1tong}) holds for $i=0$ follows from (\ref{estizon1tongoo})-(\ref{esta5}).
\end{proof}

\begin{lemma}\label{app1a}
Let the functions $z_j$ be  defined in \eqref{capitalzeta0}-\eqref{chico2}, and $\mu=\mu_k$ be defined in \eqref{parameters} and satisfies (\ref{parametersas}).  It holds that, for $i\neq j$,
\begin{align}\label{estizon1}
  \int_{\R^n} |Q (y)|^{p-1} z_i (y) z_j (y) \, dy
=
\left\{
\begin{array}{ll}
(k+1)\tilde{c} +   \left\{
\begin{array}{ll}
O(k^{ 1-\frac{n}{q} }), \ &\mbox{ if }\ n\geq 4,\\[1mm]
O(|\log k|^{- 1 }), \ &\mbox{ if }\ n=3,
\end{array}
\right. \   &\mbox{for }\ i=1,\ j=n+2,\\[3mm]
(k+1)\tilde{c}+    \left\{
\begin{array}{ll}
O(k^{ 1-\frac{n}{q} }), \ &\mbox{ if }\ n\geq 4,\\[1mm]
( |\log k|^{- 1 }), \ &\mbox{ if }\ n=3,
\end{array}
\right. \   &\mbox{for}\ i=2,\ j=n+3,\\[3mm]
O(\mu^{\frac{n-2}{2}}), \   &\mbox{ otherwise},
\end{array}
\right.
\end{align}
where  $q\in(\frac{n}{2},n)$  and $\tilde{c}$ is a positive constant, which is defined in (\ref{defctilde}).
\end{lemma}

\begin{proof}
We will only consider the case $i=1,j=n+2$, and the case $i=2,j=n+3$ can be proved in a similar way. Moreover, we omit the proof for the other cases, which can be obtained easily by using the definition of $z_i$ and the symmetry. We have
\begin{align*}
&\int_{\R^n} |Q (y)|^{p-1} z_1 (y) z_{n+2} (y) \, dy  \nonumber\\
=&\int_{\R^n} |Q (y)|^{p-1}  z_1(y)\Big(-2 y_1 z_0 (y) + |y|^2 z_1 (y)\Big)  dy  \nonumber\\
=&-2\int_{\R^n} |Q (y)|^{p-1} y_1 {\partial Q(y)\over \partial y_1 } \Big({n-2 \over 2} Q(y) + \nabla Q (y) \cdot y \Big)dy
 +\int_{\R^n} |Q (y)|^{p-1} |y|^2 \Big({\partial Q(y)\over \partial y_1 }\Big)^2   \,dy\nonumber\\
:=&L_1+L_2.
\end{align*}
Since
\begin{align*}
L_1=& -2\int_{\R^n} |Q (y)|^{p-1} y_1 {\partial Q(y)\over \partial y_1 } \Big({n-2 \over 2} Q(y) + \nabla Q (y) \cdot y \Big)dy\\
= &-2\int_{\R^n} \Big|U(y) - \sum_{j=1}^k U_j (y)+\tilde  \phi (y)\Big|^{p-1} \Big(y_1\partial_{y_1}U(y) - \sum_{j=1}^k y_1\partial_{y_1}U_j (y)+y_1\pi_1(y)\Big)\\
 &\quad \times \Big[{n-2 \over 2} U(y) + \nabla U (y) \cdot y- \sum_{j=1}^k\Big({n-2 \over 2} U_j(y) + \nabla U_j (y) \cdot y\Big) +\pi_0(y)\Big]dy\\
 = &-2\int_{\R^n} \Big[|U(y)|^{p-1} + \sum_{j=1}^k |U_j (y)|^{p-1}+|\tilde  \phi (y)|^{p-1}+|U(y)|^{\gamma}|\sum\limits_{j=1}^kU_j(y)+\tilde{\phi}(y)|^{p-1-\gamma}\Big]\\
&\quad \times \Big[y_1\partial_{y_1}U(y) - \sum_{j=1}^k y_1\partial_{y_1}U_j (y)+y_1\pi_1(y)\Big]\\
&\quad \times \Big[{n-2 \over 2} U(y) + \nabla U (y) \cdot y- \sum_{j=1}^k\Big({n-2 \over 2} U_j(y) + \nabla U_j (y) \cdot y\Big) +\pi_0(y)\Big]dy\\
=&-2\int_{\R^n}  |U(y)|^{p-1}y_1\partial_{y_1}U(y)\Big[{n-2 \over 2} U(y) + \nabla U (y) \cdot y\Big]dy \\
&-2\sum\limits_{j=1}^k\int_{\R^n}  |U_j(y)|^{p-1}y_1\partial_{y_1}U_j(y)\Big[{n-2 \over 2} U_j(y) + \nabla U_j(y) \cdot y\Big]dy \\
&\ \ + \left\{
\begin{array}{ll}
O(k^{ 1-\frac{n}{q} }), \  &\mbox{ if }\ n\geq 4,\\[1mm]
O(|\log k|^{- 1 }), \  &\mbox{ if }\  n=3.
\end{array}
\right.\\
=&(n-2)^2\alpha_n^{p+1}\int_{\R^n}\frac{y_1^2(1-|y|^2)}{(1+|y|^2)^{n+2}}dy\\
&+(n-2)^2\alpha_n^{p+1}\sum\limits_{j=1}^k \mu_k^n\int_{\R^n} \frac{y_1(y_1-\xi_{j,1})(\mu_k^2+|y-\xi_j|^2-2\sum_{i=1}^n y_i(y-\xi_j)_i)}{(\mu_k^2+|y-\xi_j|^2)^{n+2}}\\
&\ \ +  \left\{
\begin{array}{ll}
O(k^{ 1-\frac{n}{q} }), \  &\mbox{ if }\  n\geq 4,\\[1mm]
O(|\log k|^{- 1 }), \  &\mbox{ if }\  n=3,
\end{array}
\right.
\end{align*}
where
\begin{align*}
 & \sum\limits_{j=1}^k \mu_k^n\int_{\R^n} \frac{y_1(y_1-\xi_{j,1})(\mu_k^2+|y-\xi_j|^2-2\sum_{i=1}^n y_i(y-\xi_j)_i)}{(\mu_k^2+|y-\xi_j|^2)^{n+2}}dy\\
=&\sum\limits_{j=1}^k \mu_k^n\int_{\R^n} \frac{(\mu_kz_1+\xi_{j,1})\mu_k z_1(\mu_k^2+\mu_k^2|z|^2-2\sum_{i=1}^n (\mu_kz_i+\xi_{j,i})\mu_kz_i)}{(\mu_k^2+\mu_k^2|z|^2)^{n+2}}dy\\
=&\sum\limits_{j=1}^k \mu_k^{-2} \int_{\R^n} \frac{(\mu_k z_1+ \xi_{j,1})  z_1(1+|z|^2-2\sum_{i=1}^n (\mu_kz_i+ \xi_{j,i}) z_i)}{(1+ |z|^2)^{n+2}}dz\\
=&\sum\limits_{j=1}^k   \int_{\{|z|\leq \frac{1}{2\mu_k}\}} \frac{(  z_1+ \frac{1}{\mu_k}\xi_{j,1})  z_1(1+|z|^2-2\sum_{i=1}^n ( z_i+\frac{1}{\mu_k} \xi_{j,i}) z_i)}{(1+ |z|^2)^{n+2}}dz\\
&+\sum\limits_{j=1}^k   \int_{\{|z|\geq \frac{1}{2\mu_k}\}} \frac{(  z_1+\frac{1}{\mu_k} \xi_{j,1})  z_1(1+|z|^2-2\sum_{i=1}^n ( z_i+ \frac{1}{\mu_k}\xi_{j,i}) z_i)}{(1+ |z|^2)^{n+2}}dz\\
=&O\Big(\sum\limits_{j=1}^k   \int_{\{|z|\leq \frac{1}{2\mu_k}\}}\frac{\xi_{j,1}^2}{\mu_k^2}\frac{z_1^2}{(1+ |z|^2)^{n+2}}dz\Big)+\sum\limits_{j=1}^k   \int_{\{|z|\geq \frac{1}{2\mu_k}\}} \frac{  z_1^2(1-|z|^2)}{(1+ |z|^2)^{n+2}}dz \\
=&k\int_{\R^n} \frac{  y_1^2(1-|y|^2)}{(1+ |y|^2)^{n+2}}dz+O(k\mu_k^{n+2}).
\end{align*}
Thus we find
\begin{align*}
L_1=&(k+1)\int_{\R^n} (n-2)^2\alpha_n^{p+1}\frac{  y_1^2(1-|y|^2)}{(1+ |y|^2)^{n+2}}dz
 +  \left\{
\begin{array}{ll}
O(k^{ 1-\frac{n}{q} }), \  &\mbox{ if }\  n\geq 4,\\[1mm]
O(|\log k|^{- 1 }), \   &\mbox{ if }\  n=3.
\end{array}
\right.
\end{align*}
Moreover
\begin{align*}
L_2=&\int_{\R^n} |Q (y)|^{p-1} |y|^2 \Big({\partial Q(y)\over \partial y_1 }\Big)^2   \,dy\\
=&\int_{\R^n} \Big|U(y) - \sum_{j=1}^k U_j (y)+\tilde  \phi (y)\Big|^{p-1} |y|^2\Big( \partial_{y_1}U(y) - \sum_{j=1}^k  \partial_{y_1}U_j (y)+ \pi_1(y)\Big)^2   \,dy\\
=&\int_{\R^n} \Big[|U(y)|^{p-1} + \sum_{j=1}^k |U_j (y)|^{p-1}+|\tilde  \phi (y)|^{p-1}+|U(y)|^{\gamma}|\sum\limits_{j=1}^kU_j(y)+\tilde{\phi}(y)|^{p-1-\gamma}\Big]\\
&\quad \times |y|^2\Big[ (\partial_{y_1}U(y))^2 + \sum_{j=1}^k ( \partial_{y_1}U_j (y))^2+ (\pi_1(y))^2+2\partial_{y_1}U(y) \Big(\sum\limits_{j=1}^kU_j(y)+\tilde{\phi}(y)\Big)\Big]   \,dy\\
=&\int_{\R^n}  |U(y)|^{p-1}|y|^2(\partial_{y_1}U(y))^2 dy+\sum\limits_{j=1}^k\int_{\R^n}  |U_j(y)|^{p-1}|y|^2(\partial_{y_1}U_j(y))^2 dy+O(k\mu_k^{\frac{n-2}{2}})\\
=&(k+1)\int_{\R^n} (n-2)^2\alpha_n^{p+1}\frac{  y_1^2 |y|^2 }{(1+ |y|^2)^{n+2}}dz +   \left\{
\begin{array}{ll}
O(k^{ 1-\frac{n}{q} }), \ &\mbox{ if }\ n\geq 4,\\[1mm]
O(|\log k|^{- 1 }), \  &\mbox{ if }\  n=3.
\end{array}
\right.
\end{align*}
Therefore, we obtain
\begin{align*}
& \int_{\R^n} |Q (y)|^{p-1} z_1 (y) z_{n+2} (y) \, dy \\
 =& (k+1) (n-2)^2\alpha_n^{p+1}\int_{\R^n}\frac{  y_1^2   }{(1+ |y|^2)^{n+2}}dz +  \left\{
\begin{array}{ll}
O(k^{ 1-\frac{n}{q} }), \ &\mbox{ if }\ n\geq 4,\\[1mm]
O(|\log k|^{- 1 }), \ &\mbox{ if }\ n=3,
\end{array}
\right.\\
  =&(k+1) \int_{\mathbb{R}^n}U(y)^{p-1}\mathcal{Z}_{1}(y)^2dy+O(k\mu_k^{\frac{n-2}{2}}) +   \left\{
\begin{array}{ll}
O(k^{ 1-\frac{n}{q} }), \ &\mbox{ if }\ n\geq 4,\\[1mm]
O(|\log k|^{- 1 }), \ &\mbox{ if }\ n=3,
\end{array}
\right.\\
 =&(k+1)\tilde{c} +   \left\{
\begin{array}{ll}
O(k^{ 1-\frac{n}{q} }), \ &\mbox{ if }\ n\geq 4,\\[1mm]
O(|\log k|^{- 1 }), \ &\mbox{ if }\ n=3,
\end{array}
\right.
\end{align*}
where $\mathcal{Z}_{1}(y)=\frac{\partial U(y)}{\partial y_1}$.
\end{proof}

\begin{lemma}\label{app2} For any constant $a>0$, there exists $C>0$ such that
\begin{equation}\label{ape3}
\int_{\R^n} {1\over |y-z|^{n-2}} \, {1\over (1+ |z| )^{2+a} } \, dz \leq {C \over 1+ |y|^a}.
\end{equation}
For any $0<b<n$, there exists a constant $C> 0$ such that
\begin{align}\label{ape4}
\int_{B(0,1)} {1\over |y-z|^{n-2}} \, {1\over |z|^{n-b}} \, dz \leq {C \over |y|^{n-2-b}}.
\end{align}
\end{lemma}

\begin{proof}
\noindent
{\it Proof of \eqref{ape3}}. \ \ We just need to give the estimate for $|y|\geq2$. Let $d=\frac{1}{2}|y|$.

\noindent For $z\in B_d(0)$, we have $|z|\leq d$ and $|y-z|\geq |y|-|z|\geq d$, then
\begin{align*}
\int_{B_d(0)} {1\over |y-z|^{n-2}} \, {1\over (1+ |z| )^{2+a} } \, dz \leq& \frac{C}{d^{n-2}}\int_{B_d(0)}  {1\over (1+ |z| )^{2+a} } \, dz\\
\leq &\frac{C}{d^{n-2}}d^{n-2-a}\leq \frac{C}{d^a}.
\end{align*}
For $z\in B_d(y)$, we have $|y-z|\leq d$ and $|z|=|y-(y-z)|\geq |y|-|y-z|\geq d$, we then have
\begin{align*}
\int_{B_d(y)} {1\over |y-z|^{n-2}} \, {1\over (1+ |z| )^{2+a} } \, dz \leq& \frac{C}{d^{2+a}}\int_{B_d(y)} {1\over |y-z|^{n-2}} \, dz \leq \frac{C}{d^a}.
\end{align*}
For $z\in\R^n\backslash(B_d(0)\cup B_d(y))$, we have $|z-y\geq\frac{1}{2}|y|$ and $|z|\geq\frac{1}{2}|y|$. If $|z|\geq 2|y|$, we have $|z-y|\geq |z|-|y|\geq\frac{1}{2}|z|$. Thus
\begin{align}\label{appc1}
{1\over |y-z|^{n-2}} \, {1\over (1+ |z| )^{2+a} }  \leq \frac{C}{|z|^{n-2} (1+ |z| )^{2+a} }.
\end{align}
If $|z|\leq 2|y|$, then
\begin{align}\label{appc2}
{1\over |y-z|^{n-2}} \, {1\over (1+ |z| )^{2+a} }  \leq \frac{C}{|y|^{n-2} (1+ |z| )^{2+a} }\leq \frac{C'}{|z|^{n-2} (1+ |z| )^{2+a} }.
\end{align}
From (\ref{appc1}) and (\ref{appc2}), we get that for $z\in\R^n\backslash(B_d(0)\cup B_d(y))$,
\begin{align*}
{1\over |y-z|^{n-2}} \, {1\over (1+ |z| )^{2+a} }   \leq \frac{C}{|z|^{n-2} (1+ |z| )^{2+a} }.
\end{align*}
Then
\begin{align*}
\int_{\R^n\backslash(B_d(0)\cup B_d(y))} {1\over |y-z|^{n-2}} \, {1\over (1+ |z| )^{2+a} } \, dz \leq& \int_{\R^n\backslash(B_d(0)\cup B_d(y))}\frac{C}{|z|^{n-2} (1+ |z| )^{2+a} }\, dz \leq \frac{C}{d^a}.
\end{align*}

\medskip
\noindent
{\it Proof of \eqref{ape4}}. \ \ Assume $y= r \hat y$, with $r=|y|$. Direct computation gives
\begin{eqnarray*}
\int_{B(0,1)} {1\over |y-z|^{n-2}} \, {1\over |z|^{n-b}} \, dz  &=& {1\over r^{n-2}} \int_{B(0,1)} {1\over |\hat y -{ z \over r} |^{n-2}} \, {1\over |z|^{n-b}} \, dz \\
& & (w= {z\over r}) \\
&=&{1\over r^{n-2-b}} \int_{B(0, {1\over r} )} {1\over |\hat y -w|^{n-2}} \, {1\over |w|^{n-b}} \, dz\\
&\leq & {C\over |y|^{n-2-b}} ,
\end{eqnarray*}
where
$$
C= \sup_{e \in S^{n}} \int_{\R^n} {1\over |e -w|^{n-2}} \, {1\over |w|^{n-b}} \, dz.
$$

\end{proof}

\end{document}